\documentclass[11pt]{article}
\usepackage[linesnumbered,ruled,vlined]{algorithm2e}
\usepackage{amsfonts}
\usepackage{amsmath}
\usepackage{amssymb}
\usepackage{amsthm}
\usepackage{bbm}
\usepackage{cases}
\usepackage{bm}
\usepackage{float}
\usepackage{braket}
\usepackage{graphicx}
\usepackage{subcaption}
\usepackage{caption}
\usepackage[title]{appendix}
\usepackage{dsfont}
\usepackage{booktabs}
\usepackage{makecell}
\usepackage{authblk}
\usepackage{enumitem}
\usepackage{graphicx}
\usepackage{hyperref}
\usepackage{cleveref}
\usepackage{lineno}
\usepackage{mathrsfs}
\usepackage{mathtools}
\usepackage{multirow}
\usepackage{multicol}
\usepackage{natbib}
\usepackage{ulem}
\usepackage{xcolor}
\usepackage[version=4]{mhchem}

\allowdisplaybreaks[4]

\numberwithin{equation}{section}
\theoremstyle{plain}
\newtheorem{theorem}{Theorem}[section]  
\newtheorem{lemma}{Lemma}[section]
\newtheorem{corollary}{Corollary}[section]

\newtheorem{assumption}{Assumption}[section]

\theoremstyle{plain}
\newtheorem{definition}{Definition}[section] %

\theoremstyle{remark}
\newtheorem{remark}{Remark}[section]

\hypersetup{colorlinks,breaklinks}

\newcommand{\dd}{\mathrm{\,d}}

\newcommand{\eps}{\varepsilon}
\newcommand{\F}{\mathrm{F}}

\newcommand{\st}{\mathrm{s.~t.}}
\newcommand{\St}{\mathrm{St}}

\newcommand{\trace}{\mathrm{Tr}}

\newcommand{\calI}{\mathcal{I}}

\newcommand{\calL}{\mathcal{L}}

\newcommand{\calO}{\mathcal{O}}

\newcommand{\calT}{\mathcal{T}}

\newcommand{\calY}{\mathcal{Y}}

\newcommand{\N}{\mathbb{N}}

\newcommand{\R}{\mathbb{R}}
\renewcommand{\S}{\mathbb{S}}

\newcommand{\bmu}{\boldsymbol{\mu}}

\newcommand{\bpsi}{\boldsymbol{\psi}}
\newcommand{\brho}{\boldsymbol{\rho}}

\newcommand{\rr}{\boldsymbol{r}}
\newcommand{\RR}{\boldsymbol{R}}

\newcommand{\lrbrace}[1]{\left\{#1\right\}}
\newcommand{\lrbracket}[1]{\left(#1\right)}

\newcommand{\abs}[1]{\left|#1\right|}
\newcommand{\inner}[1]{\left\langle#1\right\rangle}
\newcommand{\norm}[1]{\left\Vert#1\right\Vert}
\newcommand{\snorm}[1]{\Vert#1\Vert}

\newcommand{\xc}{\mathrm{xc}}
\newcommand{\nb}{n_{\rm{b}}}

\newcommand{\proj}{\mathrm{Proj}}

\newcommand{\projj}{{\rm{P}}_{X^{(k+1)}}^{\perp}}

\definecolor{lightblue}{rgb}{0.957,0.963,0.975}
\definecolor{midblue}{rgb}{0.937,0.943,0.965}
\definecolor{deepblue}{rgb}{0.325,0.427,0.569}
\definecolor{blocktitleblue}{rgb}{0.225,0.427,0.669}
\definecolor{lightred}{rgb}{0.996,0.969,0.969}
\definecolor{midred}{rgb}{0.976,0.949,0.949}
\definecolor{deepred}{rgb}{0.686,0.133,0.098}
\definecolor{deepgreen}{rgb}{0,0.5,0}
\definecolor{halfgray}{gray}{0.55}

\definecolor{lightpurple}{rgb}{0.978,0.978,1.0}
\definecolor{deeppurple}{rgb}{0.353,0.275,0.478}

\DeclareMathOperator*{\Span}{span}
\DeclareMathOperator*{\argmin}{arg\,min}

\title{
   \textbf{A Damped Subspace Splitting Algorithm for Constrained Density Functional Theory}\footnote{The work of Xin Liu was supported in part by the National Natural Science Foundation of China (No. 12125108). The work of Yukuan Hu has received funding from the European Research Council (ERC) under the European Union’s Horizon 2020 research and innovation program (grant agreement EMC2 No. 810367). The work of Guanghui Hu was supported by The Science and Technology Development Fund, Macao SAR (Nos. 0068/2024/RIA1, 001/2024/SKL), National Natural Science Foundation of China (No. 11922120), MYRG of University of Macau (No. MYRG-CRG2024-00042-FST).}
}

\author[1]{Yuanming Su}
\author[2]{Yukuan Hu}
\author[1]{Xin Liu}
\author[3]{Guanghui Hu}

\affil[1]{{\small{State Key Laboratory of Mathematical Sciences, Academy of Mathematics and Systems Science, Chinese Academy of Sciences, Beijing 100190, China and School of Mathematical Sciences, University of Chinese Academy of Sciences, Beijing 100049, China}}}
\affil[2]{{\small{CERMICS, CNRS, \'Ecole des Ponts, Institut Polytechnique de Paris, 6-8 avenue Blaise Pascal, Cit\'e Descartes, 77455 Marne-la-Vall\'ee, France}}}
\affil[3]{\small{State Key Laboratory of Internet of Things for Smart City and Department of Mathematics, University of Macau, Macao, China}}

\begin{document}

\maketitle

\begin{abstract}
Constrained density functional theory (CDFT) provides a powerful framework for describing electronically excited and charge-localized states, which underlie a broad range of physical and chemical phenomena. However, the discretized optimization problems arising from CDFT calculations remain challenging, owing to the presence of both the Stiefel manifold constraint and additional nonconvex quadratic constraints. Existing algorithms either fail to enforce the quadratic constraints with high accuracy or face convergence issues due to double-loop iterative structures. In this paper, we first derive a subspace-splitting reformulation that decouples the two groups of constraints, by exploiting the inherent rotation invariance and introducing a nonlinear subspace alignment constraint. Based on this reformulation, we propose a single-loop damped alternating direction method of multipliers, called DASSP. To the best of our knowledge, DASSP is the first algorithm for CDFT calculations with rigorous convergence guarantees. Each iteration of DASSP comprises a spectral minimization step, a projected gradient step, and a damped dual ascent step, all of which admit efficient implementations. Numerical results on synthetic and realistic CDFT problems demonstrate that DASSP attains high feasibility accuracy and exhibits favorable efficiency without compromising robustness. We expect that this work will pave the way toward reliable and efficient large-scale CDFT applications.
\end{abstract}

\section{Introduction}\label{sec:introduction}







\par Building on the seminal works of the 1960s \cite{hohenberg1964inhomogeneous,kohn1965self}, density functional theory (DFT) has been the workhorse for first-principles simulations of ground-state properties of molecules, materials, and nanosystems \cite{burke2012perspective}. 
While the ground states are recognized to be the most stable in nature, excited states play a crucial role in a wide range of physical and chemical processes and are essential for understanding phenomena such as light absorption, emission, and energy transfer \cite{balzani2014photochemistry,turro2009principles}. To describe electronically excited or charge-localized states within the DFT framework, several extensions and approaches have been developed, including constrained DFT (CDFT) \cite{dederichs1984ground,kaduk2012constrained}. 

\par This paper is concerned with the optimization problems arising from CDFT, which have received little attention from the mathematics community, in stark contrast to the DFT \cite{cances2023density,lin2019numerical}. In its general form \cite{gonze2022constrained,kaduk2012constrained}, for an $N$-electron system, the CDFT problem amounts to solving\footnote{We consider real Hamiltonians without spin-orbit coupling or magnetic fields; therefore, the spinor orbitals can be assumed to be real-valued.}
\begin{align}
    \inf_{\{\bpsi_i\}_{i=1}^N}&\quad E[\brho[\{\bpsi_i\}_{i=1}^N]],\nonumber\\
    \st& \quad \int_{\R^3}\bpsi_{i}^{\top}(\rr)\bpsi_{j}(\rr)\dd \rr=\delta_{ij},~~i,j=1,\ldots,N,\label{eqn:CDFT continuous general} \\
    &\quad \sum_{\sigma,\sigma'\in\{\uparrow,\downarrow\}}\int_{\R^3} w_j^{\sigma,\sigma'}\rho^{\sigma,\sigma'}[\{\bpsi_i\}_{i=1}^N](\rr)\dd \rr=N_j,~~j=1,\ldots,m.\nonumber   
\end{align}
Here, $\{\bpsi_i:=(\psi_{i}^{\uparrow},\psi_{i}^{\downarrow})^\top\}_{i=1}^N$ denote a set of orthonormal spinor orbitals, with $\psi_{i}^{\uparrow}$ and $\psi_{i}^{\downarrow}$ the spin-up and spin-down components, respectively; $\brho:=\begin{bsmallmatrix}\rho^{\uparrow,\uparrow} & \rho^{\uparrow,\downarrow}\\ \rho^{\downarrow,\uparrow} & \rho^{\downarrow,\downarrow}\end{bsmallmatrix}$ refers to the spin-density matrix, defined by $\brho=\sum_{i=1}^N\bpsi_i\bpsi_i^\top$; $E$ represents the Kohn-Sham energy functional in $\brho$ \cite{kohn1965self}\footnote{We consider the energy functional in density, allowing the use of all local and semi-local as well as some non-local exchange-correlation functionals. See the perspective \cite{burke2012perspective} for more details.}; $\{w_j^{\sigma,\sigma'}\}_{\sigma,\sigma'}$ are the weight functions associated with the $j$-th constraint and $N_j\in \R$ ($j=1,\ldots,m$). 
Such a formulation encompasses several practical situations \cite{gonze2022constrained,kaduk2012constrained}. For example, to constrain the charge of a spatial fragment, one may choose the weight functions such that $w_j^{\uparrow,\uparrow}=w_j^{\downarrow,\downarrow}=w_j$ and $w_j^{\uparrow,\downarrow}=w_j^{\downarrow,\uparrow}\equiv0$, where $w_j$ is a smoothed indicator function of the fragment. This yields a constraint on the local charge density, useful for characterizing charge-transfer states \cite{wu2005direct,wu2006constrained}. 
In noncollinear DFT, it is common to constrain local atomic magnetic moments in order to investigate energy landscapes of magnetic materials and to construct machine-learned surrogates \cite{ma2015constrained,yu2024physics,yu2024spin,zheng2026integrating}. This can be achieved by defining the weight functions using Pauli matrices and smoothed indicator functions localized around atoms. In these typical applications, $m$ is at most on the order of $\calO(N_{\rm nuc})$, with $N_{\rm nuc}\in\N$ the number of nuclei. We also note that the CDFT reduces to the DFT when $m$ equals zero.

\par The continuous formulation \eqref{eqn:CDFT continuous general} can be discretized upon choosing suitable basis sets; interested readers are referred to Appendix \ref{appsec:UKS functional}. 
The resulting optimization problem is of the form
\begin{equation}
    \min_{X}~f(XX^\top),~~\st~X\in\St_p(\R^n),~\trace(X^\top W_jX)=b_j,~~j=1,\ldots,m,
    \label{eqn:CDFT discrete general}
\end{equation}
which involves the Stiefel manifold $\St_p(\R^n):=\{X\in\R^{n\times p}: X^\top X=I_p\}$ together with $m$ additional nonconvex quadratic constraints defined by $W_j\in\S^n$ and $b_j\in\R$ ($j=1,\ldots,m$). For simplicity, we impose the following blanket assumption.
\begin{assumption}[Blanket assumption]\label{asp:LICQ}\mbox{}
    \begin{enumerate}[label={\rm(A\arabic*)}]
        \item Problem \eqref{eqn:CDFT discrete general} has at least one feasible point.

        \item\label{asp:LICQ-(2)} The (Riemannian) linear independence constraint qualification is fulfilled at any local minimizer $X$ of problem \eqref{eqn:CDFT discrete general}, namely, $\{(I_n-XX^\top)W_jX\}_{j=1}^m$ are linearly independent.
    \end{enumerate}
\end{assumption}

\begin{remark}
    Assumption \ref{asp:LICQ} \ref{asp:LICQ-(2)} implies the necessity of the Karush-Kuhn-Tucker (KKT) stationary conditions for any local minimizer of problem \eqref{eqn:CDFT discrete general} {\rm\cite{bergmann2019intrinsic}}. 
\end{remark}
\noindent 
In the context of CDFT, the matrix $X$, the objective function $f$, and the matrices $\{W_j\}_{j=1}^m$ are derived from the discretizations of spinor orbitals, energy functional, and weight functions, respectively. From the optimization point of view, the main difficulty in solving problem \eqref{eqn:CDFT discrete general} lies in the coupling of manifold constraints with additional quadratic constraints. If the latter are missing, problem \eqref{eqn:CDFT discrete general} becomes an optimization problem on the Stiefel manifold and a wide range of optimization algorithms are applicable; see \cite{absil2008optimization,hu2020brief} and the references therein. One could also perform the self-consistent field (SCF) fixed-point iteration, which is a conventional option in DFT calculations \cite{cances2023density,lin2019numerical}. The presence of quadratic constraints spoils such geometrical structure. 

\par There are mainly two classes of methods for solving problem \eqref{eqn:CDFT discrete general} in the CDFT community. The most straightforward one gets rid of the $m$ quadratic constraints by adding penalty terms to the objective function \cite{ma2015constrained}. With a fixed penalty parameter $\sigma>0$, the quadratic penalty problem, for instance,
\begin{equation}
    \min_X~f(XX^\top)+\sigma\sum_{j=1}^m\lrbracket{\trace(X^\top W_jX)-b_j}^2,~~\st~~X\in\St_p(\R^n)
    \label{eqn:CDFT quadratic penalty}
\end{equation}
could be solved by either optimization algorithms on the Stiefel manifold or by the SCF iteration mentioned above \cite{absil2008optimization,cances2023density,hu2020brief,lin2019numerical}. Such a methodology has been implemented by materials simulation software packages such as VASP \cite{kresse1996efficiency,kresse1996efficient}. However, since the exactness of penalty is not guaranteed\footnote{In fact, the exactness is not likely to hold; see our numerical results in Section \ref{subsec:inexactness of quadratic penalty}.}, one needs to choose a large penalty parameter to reduce the feasibility violation, while an excessively large penalty parameter can lead to numerical instability. 

\par The second class consists of double-loop methods that achieve feasibility by solving a sequence of simpler subproblems at each iteration \cite{cai2023self,gonze2022constrained,wu2005direct,zheng2026integrating}. The KKT conditions for problem \eqref{eqn:CDFT discrete general} are
\begin{align}
    F(X,\bmu)X:=\lrbracket{\nabla f(XX^\top)+\sum\nolimits_{j=1}^m\mu_jW_j}X=X\Sigma,\label{eqn:CDFT stationarity conditions nepv}\\[0.2cm]
    \trace(X^\top W_jX)=b_j,~~j=1,\ldots,m,\nonumber\\[0.2cm]
    X\in\St_p(\R^n),~~\Sigma\in\S^{p},~~\mu_j\in\R,~~j=1,\ldots,m,\nonumber
\end{align}
where $F(X,\bmu)$ is called the discretized effective Hamiltonian in related fields, $\bmu:=(\mu_1,\ldots,\mu_m)^{\top}$ and $\Sigma$ are introduced as the Lagrange multipliers corresponding to the quadratic and manifold constraints, respectively. Given $\bmu^{(k)}$, these methods first update $X^{(k)}$ to an intermediate iterate $\tilde X^{(k+1)}$ by taking the lowest  $p$ eigenvectors of $F(X^{(k)},\bmu^{(k)})$\footnote{In practical implementations, direct inversion in the iterative subspace (DIIS) extrapolation is often applied during the outer loop, further modifying the effective operator used to compute $\tilde X^{(k+1)}$ \cite{pulay1980convergence,pulay1982diis}.}, which may violate the quadratic constraints. The next iterate $(X^{(k+1)},\bmu^{(k+1)})$ is then obtained by approximately solving the subproblem 
$$\min_X~\trace\lrbracket{X^\top\nabla f\big(\tilde X^{(k+1)}(\tilde X^{(k+1)})^\top\big)X},~~\st~X\in\St_p(\R^n),~\trace(X^\top W_jX)=b_j,~~j=1,\ldots,m,$$
using a dual Newton method initialized at  $\bmu^{(k)}$. 
The Jacobian required in the inner Newton steps can be derived from first-order perturbation theory \cite{wu2005direct}. Note that the subproblem is constructed by freezing the nonlinear dependence of $f$ on $X$, 
which is progressively corrected by the outer-loop intermediate step. Such a double-loop (or nested)
methodology may avoid issues related to penalty methods, and has been widely integrated into off-the-shelf codes, e.g., \textsc{Siesta} \cite{garcia2020siesta} and \textsc{CP2K} \cite{kuhne2020cp2k}. Nevertheless, its theoretical properties remain largely unclear. In fact, due to the nonlinearity of $f$, equation \eqref{eqn:CDFT stationarity conditions nepv} does not necessarily imply that $X$ spans the lowest $p$-dimensional invariant subspace of $F(X,\bmu)$ (see \cite{liu2015analysis} for an example), whereas both $\tilde X^{(k+1)}$ and $X^{(k+1)}$ are computed via minimization strategies, resulting in a potential mismatch. 
There would be more subtleties when taking the varying $\bmu$ into account. Moreover, the notion of the ``lowest'' $p$-dimensional eigenspace is no longer uniquely defined if the eigengap condition $\lambda_{n-p}(F(X,\bmu))>\lambda_{n-p+1}(F(X,\bmu))$ does not hold, which can lead to oscillations in the iterate sequence. Finally, the overall efficiency and convergence depend critically on the performance of inner Newton steps. 
In summary, an efficient yet robust numerical algorithm for problem \eqref{eqn:CDFT discrete general} is still lacking. 

\medskip

\par\noindent\textbf{Contributions.} 
To address these gaps, the main contributions of this paper are threefold:
\begin{itemize}
    \item We propose a subspace-splitting reformulation of problem \eqref{eqn:CDFT discrete general} that exploits the inherent rotation invariance of CDFT to decouple the Stiefel manifold constraint from the nonconvex quadratic constraints. Based on this separable structure, we design a single-loop damped subspace splitting algorithm (DASSP) for the reformulated problem. 
    

    \item We establish the global convergence of DASSP to approximate KKT points under suitable conditions. To the best of our knowledge, our work provides the first algorithm for CDFT calculations with rigorous convergence guarantees, distinguishing it from the two existing classes of methods.

    \item We demonstrate through extensive experiments on realistic CDFT problems that DASSP achieves superior feasibility accuracy compared to quadratic penalty methods. Furthermore, DASSP exhibits higher computational efficiency than the double-loop method, particularly in systems where eigengap issues typically destabilize standard SCF iterations, without compromising robustness.
    
\end{itemize}

We note that the proposed DASSP can be viewed as a damped alternating direction method of multipliers (ADMM). The established theoretical convergence is also worth mentioning compared with recent advances on ADMM-type methods for nonconvex optimization problems, which we briefly review below.

\medskip

\par \noindent \textbf{Literature review related to ADMM.} Historically, ADMM dates back to the seminal works in the mid-1970s \cite{gabay1976dual,glowinski1975approximation} and is deeply rooted in the classical augmented Lagrangian methods (ALM) \cite{hestenes1969multiplier,powell1969method}. While the literature on ADMM and its variants for convex problems is voluminous (see recent surveys \cite{boyd2011distributed,han2022survey} and the references therein), the convergence theory for nonconvex and nonlinearly constrained problems is still much less complete. A central theoretical challenge in this context is to ensure the boundedness of the dual iterate sequence \cite{hallak2023adaptive}. This difficulty is further exacerbated when the constraints are allowed to be nonlinear, as in our case. 

\par Many existing analyses assume the boundedness directly \cite{bolte2018nonconvex,cohen2022dynamic,el2025convergence,hallak2023adaptive,hien2024inertial}. There are a few exceptions \cite{sun2024dual,zhu2024first}. However, their analyses rely on assumptions that are not aligned with our setting (cf. Eq. \eqref{eqn:CDFT discrete splitting} later): the objective is bounded from below and above \cite{sun2024dual} (or coercive and partially Lipschitz continuous \cite{zhu2024first}), the nonlinear constraints are Lipschitz continuous \cite{sun2024dual,zhu2024first} and possibly norm-bounded \cite{sun2024dual}, along with an additional image-inclusion condition on the constraints \cite{zhu2024first}. 
Another related line of work concerns Riemannian ADMM \cite{deng2025adaptive,kovnatsky2016madmm,lai2014splitting,li2025riemannian,zhang2020primal}, which leverages manifold structures explicitly but assumes that any non-manifold constraints are linear and that either the objective is partially Lipschitz continuous or at least one variable block is unconstrained. In our analysis, we address the boundedness issue of the dual iterate sequence by adopting damped dual updates, combined with a careful construction of a surrogate sequence. 

\medskip

\par\noindent\textbf{Organization.} In Section \ref{sec:algorithmic developments}, we introduce the subspace-splitting reformulation of problem \eqref{eqn:CDFT discrete general} and develop the single-loop DASSP. In Section \ref{sec:convergence analysis}, we establish the convergence properties of DASSP. Numerical experiments on synthetic and realistic CDFT problems are reported in Section \ref{sec:numerical experiments}. Finally, Section \ref{sec:conclusions} concludes the paper and discusses future directions.

\medskip

\par\noindent\textbf{Notation.}
Let $\S^n$ be the set of $n\times n$ real symmetric matrices. For a matrix $A$, $A^{\top}$ denotes its transpose and $\trace{(A)}$ denotes its trace. For matrices $A$ and $B$ of the same size, we define the Frobenius inner product by $\inner{A,B}:=\trace{(A^{\top}B)}$, and denote by $\norm{A}_{\F}$ and $\norm{A}_{2}$ the Frobenius norm and spectral norm, respectively. For $A \in \S^n$, its eigenvalues are arranged as $\lambda_1(A)\geq \lambda_{2}(A)\geq \cdots \geq \lambda_n(A)$. The (real) Stiefel manifold is denoted by $\St_{p}(\R^n):=\{X\in \R^{n\times p}: X^{\top}X=I_p \}$. For $X\in \St_p(\R^n)$, we write $\mathrm{P}_{X}^{\perp}:=I_n-XX^{\top}$. The notation $\proj_{\mathcal{C}}(\cdot)$ stands for the Euclidean projection operator onto a closed set $\mathcal{C}$. 

\section{Algorithmic developments}\label{sec:algorithmic developments}




\par In this section, we develop the algorithmic framework for problem \eqref{eqn:CDFT discrete general}. By exploiting the inherent rotation invariance, we first derive a subspace-splitting reformulation that separates the Stiefel manifold constraint from the quadratic constraints. Based on this separable structure, we then propose a damped subspace splitting algorithm.

\par We have realized that the main difficulty in solving problem \eqref{eqn:CDFT discrete general} is due to the presence of two types of constraints for $X$. A natural idea is to introduce another block of variables and to treat the two groups separately using the alternating direction method of multipliers (ADMM) or its variants. 

\par There are various ways to split the variables, and our goal is to identify a strategy that facilitates the efficient solution of the resulting subproblems. To this end, we notice in problem \eqref{eqn:CDFT discrete general} the rotation invariance of both the objective function and the quadratic constraints, namely, they remain unchanged after replacing $X$ with $XQ$, where $Q\in\R^{p\times p}$ is any orthogonal matrix. In other words, problem \eqref{eqn:CDFT discrete general} depends only on the subspace ${\rm span}(X)$. This observation motivates us to introduce a matrix variable $Y\in\S^{n}$ representing the orthogonal projector onto ${\rm span}(X)$ directly, at the price of adding a subspace alignment constraint $Y=XX^\top$. We therefore have the following subspace-splitting reformulation of problem \eqref{eqn:CDFT discrete general}:
\begin{equation}\label{eqn:CDFT discrete splitting}
\min_{X,Y}~f(Y),~~\st~X\in\St_p(\R^{n}),~~Y\in \calY,~~Y=XX^{\top},
\end{equation}
where 
\begin{equation}
\calY:=\lrbrace{Y\in \S^n:\trace{(W_jY)}=b_j,~j=1,\ldots,m}.
\end{equation}
This reformulation decouples the two types of constraints in a clean way: the Stiefel manifold constraint remains entirely in the $X$-block, while the quadratic constraints become linear in the $Y$-block. 
We also note that subspace-splitting techniques have been employed in other settings, including distributed principal component analysis \cite{wang2023distributed,wang2024seeking} and sparse spectral clustering \cite{lu2018nonconvex}.

\par A straightforward application of ADMM with classical dual updates to problem \eqref{eqn:CDFT discrete splitting} would be problematic due to the introduced nonlinear constraints. As mentioned in Section \ref{sec:introduction}, their presence poses significant challenges when one attempts to show the boundedness of dual iterates \cite{hallak2023adaptive}. 
In this work, we propose a variant of ADMM that adopts a damped dual update strategy, called the \textit{damped subspace splitting algorithm} (DASSP).  

\par To begin with, we construct the damped augmented Lagrangian function:
\begin{equation}\label{eqn:augmented Lagrangian}
\calL_{\beta,\delta}(X,Y,\Lambda):=f(Y)+\inner{(1-\delta)\Lambda,Y-XX^{{\top}}}+\frac{\beta}{2}\norm{Y-XX^{\top}}_{\F}^{2},
\end{equation}
where $\Lambda\in \S^n$ is the Lagrange multiplier associated with the constraint $Y=XX^{\top}$, $\beta>0$ is the penalty parameter, and $\delta\in[0,1)$ is the damping parameter. 
Note that the damped augmented Lagrangian function \eqref{eqn:augmented Lagrangian} coincides with the standard one when $\delta=0$. The proposed DASSP updates the variables $X$, $Y$, and $\Lambda$ alternately.

\medskip

\par\noindent\textbf{$X$-update.} First, for fixed $Y^{(k)}$ and $\Lambda^{(k)}$, the subproblem for updating $X$ is given by
\begin{equation}\label{eqn:X subproblem}
X^{(k+1)}  := \argmin_{X\in\St_p(\R^{n})}~~\calL_{\beta,\delta}(X,Y^{(k)},\Lambda^{(k)})= \argmin_{X\in\St_p(\R^{n})}~~ -\trace{\lrbracket{X^{\top}A^{(k)}X}},
\end{equation}
where
\begin{equation}
    A^{(k)}:=Y^{(k)}+\frac{1-\delta}{\beta}\Lambda^{(k)}\in\S^n.
    \label{eqn:tmp matrix}
\end{equation}
In essence, the $X$-subproblem \eqref{eqn:X subproblem} is an eigenvalue problem of computing an orthonormal basis of the dominant $p$-dimensional subspace of $A^{(k)}$. This task can be accomplished by direct or iterative eigensolvers \cite{golub2013matrix}. 

\medskip

\par\noindent\textbf{$Y$-update.} Next, for given $X^{(k+1)}$, $Y^{(k)}\in\calY$, and $\Lambda^{(k)}$, we consider the following subproblem over the affine set $\calY$:
\begin{equation}\label{eqn:Y subproblem}
\begin{aligned}
    &~\min_{Y\in \calY}~~ \calL_{\beta,\delta}(X^{(k+1)},Y,\Lambda^{(k)})\\
    =&~\min_{Y\in\calY}~~f(Y)+\inner{(1-\delta)\Lambda^{(k)},Y-X^{(k+1)}(X^{(k+1)})^\top}+\frac{\beta}{2}\norm{Y-X^{(k+1)}(X^{(k+1)})^\top}_{\F}^{2}.
\end{aligned}   
\end{equation}
Due to the fact that $f$ could be highly nonlinear, the exact solution to problem \eqref{eqn:Y subproblem} is out of reach in general. Instead, we perform a single projected gradient step. To this end, let
\begin{equation*}
\calT_{\calY}:=\lrbrace{Z\in \S^n: \trace{(W_jZ)}=0,~~j=1,\ldots,m}
\end{equation*}
be the tangent space of the affine set $\calY$. Then the update of $Y^{(k+1)}$ takes the form
\begin{equation}
Y^{(k+1)}:=Y^{(k)}-\eta_k \proj_{\calT_{\calY}}\lrbracket{\nabla f(Y^{(k)})+(1-\delta)\Lambda^{(k)}+\beta(Y^{(k)}-X^{(k+1)}(X^{(k+1)})^\top)},
\label{eqn:Y update}
\end{equation}
where $\eta_k>0$ is the stepsize,  $\proj_{\calT_{\calY}}:\S^n\to \calT_{\calY}$ is the projection operator onto $\calT_{\calY}$, which admits a closed form: 
\begin{equation*}
\proj_{\calT_{\calY}}(S):=S+\sum_{j=1}^{m}\hat\mu_jW_j,\quad\forall~S\in\S^n,
\end{equation*}
where the coefficients $\hat{\bm{\mu}}:=(\hat\mu_1,\ldots,\hat\mu_m)^{\top}\in\R^m$ are determined by the linear system
\begin{equation}\label{eqn:projection linear system}
M\hat{\bm{\mu}}=-\bm{c},\quad \text{with}~~M_{ij}:=\trace{(W_iW_j)},~~c_i:=\trace{(W_iS)},~~i,j=1,\ldots,m.
\end{equation}
The invertibility of the matrix $M$ is implied by Assumption \ref{asp:LICQ}~\ref{asp:LICQ-(2)}. The computational cost of projection is negligible whenever $m$ is small.

\medskip

\par\noindent\textbf{$\Lambda$-update.} Finally, the Lagrange multiplier $\Lambda$ is updated in a damped dual ascent manner:
\begin{equation}\label{eqn:damped dual update}
\Lambda^{(k+1)}:=(1-\delta)\Lambda^{(k)}+\tau\beta\lrbracket{Y^{(k+1)}-X^{(k+1)}(X^{(k+1)})^\top},
\end{equation}
where $\tau\in (0,1]$ is an under-relaxation parameter. The presence of the damping parameter $\delta$ is a key feature of the proposed algorithm and plays crucial roles in both algorithmic and theoretical senses. First, the damped dual update \eqref{eqn:damped dual update} can be interpreted as a regularized version of the classical dual ascent step. 
Second,  bounding the dual sequence is often the central obstacle in the convergence analysis for nonconvex ADMM-type methods. Intuitively, the damping strategy introduces a contraction effect in equation \eqref{eqn:damped dual update}, making the dual iterates much easier to control. It also contributes to the construction of a surrogate sequence in the convergence analysis. We note that the damped dual update has recently been used in the ALM literature \cite{hajinezhad2019perturbed,kong2023accelerated,melo2020iteration} and the ADMM literature \cite{kong2024global,yang2017alternating,zhou2025perturbed} on linearly constrained nonconvex problems. The parameter $\tau$ is introduced due to technical reasons and is also observed to be important for numerical stability. Please refer to Sections \ref{sec:convergence analysis} and \ref{sec:numerical experiments} for more details. 

\medskip

\par We outline the proposed DASSP in Algorithm \ref{algn:DP-SSA}.


\normalem
\begin{algorithm}[!t]
        \caption{Damped subspace splitting algorithm (DASSP).}
        \label{algn:DP-SSA}
    \KwIn{Penalty parameter $\beta>0$, under-relaxation parameter $\tau\in(0,1]$,  damping parameter $\delta\in[0,1)$, stepsizes $\{\eta_k>0\}$, initial points $Y^{(0)}\in\S^n$, $\Lambda^{(0)}\in \S^n$.}

    Set $k:=0$.

    \lIf{$Y^{(0)}\notin \calY$}{project $Y^{(0)}$ onto the affine set $\calY$, i.e., $Y^{(0)}:=\proj_{\calY}(Y^{(0)})$}
    
    \While{the stopping criterion is not satisfied}{
Update $X^{(k+1)}$ by solving problem \eqref{eqn:X subproblem} with a direct or iterative eigensolver.

Update $Y^{(k+1)}$ by a single projected gradient step \eqref{eqn:Y update} with the stepsize $\eta_k$.

Update $\Lambda^{(k+1)}$ by damped dual ascent \eqref{eqn:damped dual update}.

Set $k:=k+1$.
    }
\KwOut{$X^{(k)}\in \St_p(\R^n)$.}
\end{algorithm}

\section{Convergence analysis}\label{sec:convergence analysis}


\par In this section, we establish that every accumulation point of the iterate sequence generated by DASSP is an approximate KKT point of problem \eqref{eqn:CDFT discrete general}, defined as follows.

\begin{definition}
For a given $\eps>0$, we say that $X\in \R^{n\times p}$ is an $\eps$-approximate KKT point for problem \eqref{eqn:CDFT discrete general} if there exist multipliers $\Sigma\in \S^p$ and $\bmu\in\R^m$ such that
\begin{equation*}
\max\lrbrace{\norm{F(X,\bmu)X-X\Sigma}_{\F},\norm{X^{\top}X-I_p}_{\F},\max_{1\leq j\leq m}\abs{\trace{(X^{\top}W_jX)}-b_j}}\leq \eps,
\end{equation*}
where $F(X,\bmu)$ is defined in equation \eqref{eqn:CDFT stationarity conditions nepv}. 
\end{definition}

\par We now list the assumptions for the convergence analysis.

\begin{assumption}\label{assumption}
Given the damping parameter $\delta\in(0,1)$, let $c$ be a constant satisfying $c>2(1-\delta)(2-\delta)/\delta$. 
\begin{enumerate}[label={\rm(A\arabic*)}]
\item\label{asp:lipschitz smooth} The objective function $f:\S^n\to \R$ is continuously differentiable and $L$-smooth, i.e., there exists an $L>0$ such that
$$
\norm{\nabla f(U)-\nabla f(V)}_{\F}\leq L\norm{U-V}_{\F}, \quad \forall~ U,V\in \S^n.
$$
\item\label{asp:eigengap} The sequence $\{A^{(k)}\}_{k\geq 0}$ defined in equation \eqref{eqn:tmp matrix} has a uniform eigengap, i.e., there exists a $\gamma>0$ such that
$$
\gamma_k:=\lambda_{p}(A^{(k)})-\lambda_{p+1}(A^{(k)})\geq \gamma, \quad \forall~k\in \N.
$$
\item\label{asp:penalty param} The penalty parameter $\beta$ satisfies $\beta>2(L+1)$, where $L$ is defined in \ref{asp:lipschitz smooth}.
\item\label{asp:step size} The stepsize $\eta_k\equiv\eta$ satisfies $\eta \in (0,2/(\beta+L))$.

\item\label{asp:under relaxation param} The under-relaxation parameter $\tau$ satisfies
\begin{equation*}
\tau \in \lrbracket{0,\frac{\delta}{2c}\min\lrbrace{\frac{1}{\beta}\lrbracket{\frac{1}{\eta}-\frac{\beta+L}{2}},\frac{\gamma}{n}}},
\end{equation*}
where $\gamma$ is defined in \ref{asp:eigengap}.
\end{enumerate}
\end{assumption}
\begin{remark}
{\rm{Assumption \ref{assumption}~\ref{asp:eigengap}}} is a standard condition commonly used in the literature to guarantee the well-posedness of the linear eigenvalue problem \eqref{eqn:X subproblem} throughout the iterations {\rm{\cite{bai2022sharp,liu2014convergence,liu2015analysis}}}. Although we cannot remove it in the current analysis, we remark that it holds when $Y^{(k)}$ and $\Lambda^{(k)}$ are sufficiently close to fixed matrices $\bar X\bar X^\top\in\S^n$ and $\bar\Lambda\in\S^n$ with $\bar X\in\St_p(\R^n)$, and $\beta$ is sufficiently large; this can be shown by using the perturbation theory of matrix eigenvalues \cite{golub2013matrix}. We also provide supporting numerical evidence in Section \ref{subsec:charge-constrained DFT}. In {\rm Assumption \ref{assumption}~\ref{asp:step size}}, the stepsize is assumed to be constant for simplicity. An extension of our analysis to cases with variable stepsizes $\eta_k\in[\varepsilon_0,2/(\beta+L)-\varepsilon_0]$ for any fixed $\varepsilon_0>0$ would be straightforward but tedious.
\end{remark}

Let $\{(X^{(k)},Y^{(k)},\Lambda^{(k)})\}$ be the sequence generated by Algorithm \ref{algn:DP-SSA}. Throughout this section, we denote the subspace alignment violation by 
\begin{equation}
    R^{(k)}:=Y^{(k)}-X^{(k)}(X^{(k)})^\top\in\S^n. 
    \label{eqn:subspace alignment violation}
\end{equation}
We first quantify the decrease in the damped augmented Lagrangian induced by primal updates.
\begin{lemma}\label{lemma:decrease in X step}
Suppose that {\rm Assumption \ref{assumption}~\ref{asp:eigengap}} holds. Then, for every $k\geq 0$,
\begin{equation*}
\calL_{\beta,\delta}(X^{(k+1)},Y^{(k)},\Lambda^{(k)})-\calL_{\beta,\delta}(X^{(k)},Y^{(k)},\Lambda^{(k)})\leq -\frac{\beta\gamma}{n}\norm{X^{(k+1)}(X^{(k+1)})^\top-X^{(k)}(X^{(k)})^\top}_{\F}^{2}.
\end{equation*}
\end{lemma}
\begin{proof}
A direct calculation yields
\begin{align*}
&~\calL_{\beta,\delta}(X^{(k+1)},Y^{(k)},\Lambda^{(k)})-\calL_{\beta,\delta}(X^{(k)},Y^{(k)},\Lambda^{(k)})\\
=&~-\beta\lrbracket{\trace{((X^{(k+1)})^\top A^{(k)}X^{(k+1)})}-\trace{((X^{(k)})^{\top}A^{(k)}X^{(k)})}}.
\end{align*}
Since $X^{(k)},X^{(k+1)}\in\St_p(\R^n)$, applying \cite[Lemma 3.1]{liu2014convergence} to $-A^{(k)}$, and using \rm{Assumption \ref{assumption}~\ref{asp:eigengap}}, we obtain
\begin{align*}
\calL_{\beta,\delta}(X^{(k+1)},Y^{(k)},\Lambda^{(k)})-\calL_{\beta,\delta}(X^{(k)},Y^{(k)},\Lambda^{(k)})&\leq -\beta\gamma\norm{X^{(k+1)}(X^{(k+1)})^\top-X^{(k)}(X^{(k)})^\top}_{2}^{2}\\
&\leq -\frac{\beta\gamma}{n}\norm{X^{(k+1)}(X^{(k+1)})^\top-X^{(k)}(X^{(k)})^\top}_{\F}^{2},
\end{align*}
which completes the proof.
\end{proof}
\begin{lemma}\label{lemma:decrease in Y  step}
Suppose that {\rm Assumption \ref{assumption}~\ref{asp:lipschitz smooth}} and \ref{asp:step size} hold. Then, for every $k\geq 0$,
\begin{equation*}
\calL_{\beta,\delta}(X^{(k+1)},Y^{(k+1)},\Lambda^{(k)})-\calL_{\beta,\delta}(X^{(k+1)},Y^{(k)},\Lambda^{(k)})\leq-\lrbracket{\frac{1}{\eta}-\frac{\beta+L}{2}}\norm{Y^{(k+1)}-Y^{(k)}}_{\F}^{2}.
\end{equation*}
\end{lemma}
\begin{proof}
For fixed $X^{(k+1)}$ and $\Lambda^{(k)}$, the gradient of the damped augmented Lagrangian  with respect to $Y$ is given by
\begin{equation*}
\nabla_{Y}\calL_{\beta,\delta}(X^{(k+1)},Y,\Lambda^{(k)})=\nabla f(Y)+(1-\delta)\Lambda^{(k)}+\beta(Y-X^{(k+1)}(X^{(k+1)})^\top),
\end{equation*}
which is $(\beta+L)$-Lipschitz continuous by Assumption \ref{assumption}~\ref{asp:lipschitz smooth}. Let
\begin{equation}\label{eqn:projected gradient}
G^{(k)}:=\proj_{\calT_{\calY}}\lrbracket{\nabla f(Y^{(k)})+(1-\delta)\Lambda^{(k)}+\beta(Y^{(k)}-X^{(k+1)}(X^{(k+1)})^\top)}.
\end{equation}
Then the $Y$-update gives $Y^{(k+1)}-Y^{(k)}=-\eta G^{(k)}\in \calT_{\calY}$.  By the standard descent lemma,
\begin{equation*}
\begin{aligned}
&~\calL_{\beta,\delta}(X^{(k+1)},Y^{(k+1)},\Lambda^{(k)})-\calL_{\beta,\delta}(X^{(k+1)},Y^{(k)},\Lambda^{(k)})\\
\leq &~\inner{\nabla_{Y}\calL_{\beta,\delta}(X^{(k+1)},Y^{(k)},\Lambda^{(k)}),Y^{(k+1)}-Y^{(k)}}+\frac{\beta+L}{2}\norm{Y^{(k+1)}-Y^{(k)}}_{\F}^{2}\\
= &~-\frac{1}{\eta}\norm{Y^{(k+1)}-Y^{(k)}}_{\F}^{2}+\frac{\beta+L}{2}\norm{Y^{(k+1)}-Y^{(k)}}_{\F}^{2}=-\lrbracket{\frac{1}{\eta}-\frac{\beta+L}{2}}\norm{Y^{(k+1)}-Y^{(k)}}_{\F}^{2}.
\end{aligned} 
\end{equation*}
This completes the proof.
\end{proof}
Next, we turn to the behavior of the dual updates. To facilitate the analysis, we introduce an auxiliary sequence
\begin{equation}\label{eqn:pre-surrogate sequence}
T_k:=\calL_{\beta,\delta}(X^{(k)},Y^{(k)},\Lambda^{(k)})-\frac{\delta(1-\delta)}{2\tau\beta}\snorm{\Lambda^{(k)}}_{\F}^{2}.
\end{equation}
The following lemma provides an upper bound for $T_{k+1}-T_k$. 
\begin{lemma}\label{lemma:warm-up Lyapunov}
Suppose that {\rm Assumption \ref{assumption}~\ref{asp:lipschitz smooth}, \ref{asp:eigengap}, and \ref{asp:step size}} hold. Then, for every $k\geq 0$,
\begin{align*}
T_{k+1}-T_k\leq & -\lrbracket{\frac{1}{\eta}-\frac{\beta+L}{2}}\norm{Y^{(k+1)}-Y^{(k)}}_{\F}^{2}-\frac{\beta\gamma}{n}\norm{X^{(k+1)}(X^{(k+1)})^\top-X^{(k)}(X^{(k)})^\top}_{\F}^{2}\\
&+\frac{(1-\delta)(1-\frac{\delta}{2})}{\tau\beta}\norm{\Lambda^{(k+1)}-\Lambda^{(k)}}_{\F}^{2}.
\end{align*}
\end{lemma}
\begin{proof}
By direct calculation, we have
\begin{equation*}
\calL_{\beta,\delta}(X^{(k+1)},Y^{(k+1)},\Lambda^{(k+1)})-\calL_{\beta,\delta}(X^{(k+1)},Y^{(k+1)},\Lambda^{(k)})=(1-\delta)\inner{\Lambda^{(k+1)}-\Lambda^{(k)},R^{(k+1)}}.
\end{equation*}
Using the dual update \eqref{eqn:damped dual update}, we obtain 
\begin{equation}
R^{(k+1)}=\frac{1}{\tau\beta}\lrbracket{\Lambda^{(k+1)}-(1-\delta)\Lambda^{(k)}}.    
\label{eqn:primal residual in dual}
\end{equation}
Hence,
\begin{equation*}
(1-\delta)\inner{\Lambda^{(k+1)}-\Lambda^{(k)},R^{(k+1)}}=\frac{1-\delta}{\tau\beta}\inner{\Lambda^{(k+1)}-\Lambda^{(k)},\Lambda^{(k+1)}-(1-\delta)\Lambda^{(k)}}.
\end{equation*}
A direct expansion gives
\begin{align*}
&~(1-\delta)\inner{\Lambda^{(k+1)}-\Lambda^{(k)},R^{(k+1)}}-\frac{\delta(1-\delta)}{2\tau\beta}\lrbracket{\snorm{\Lambda^{(k+1)}}_{\F}^{2}-\snorm{\Lambda^{(k)}}_{\F}^{2}}\\
=&~\frac{(1-\delta)(1-\frac{\delta}{2})}{\tau\beta}\norm{\Lambda^{(k+1)}-\Lambda^{(k)}}_{\F}^{2}.
\end{align*}
Combining this identity with Lemmas \ref{lemma:decrease in X step} and \ref{lemma:decrease in Y  step} yields the desired inequality.
\end{proof}

The upper bound in Lemma \ref{lemma:warm-up Lyapunov} contains a positive term arising from the dual update, which cannot be easily controlled by telescoping arguments. To handle this, we exploit the dual update \eqref{eqn:damped dual update} to relate the dual differences to the primal differences.
\begin{lemma}\label{lemma:compensation inequality}
For any $\eps_X,\eps_Y>0$, and any $k\geq 1$, the following inequality holds   
\begin{align*}
\frac{1-\delta}{2\tau\beta}\norm{\Lambda^{(k+1)}-\Lambda^{(k)}}_{\F}^{2}\leq &~\frac{1-\delta}{2\tau\beta}\norm{\Lambda^{(k)}-\Lambda^{(k-1)}}_{\F}^{2}-\lrbracket{\frac{\delta}{2\tau\beta}-\frac{1}{2\eps_{Y}}-\frac{1}{2\eps_{X}}}\norm{\Lambda^{(k+1)}-\Lambda^{(k)}}_{\F}^{2}\\
&~+\frac{\eps_{Y}}{2}\norm{Y^{(k+1)}-Y^{(k)}}_{\F}^{2}+\frac{\eps_{X}}{2}\norm{X^{(k+1)}(X^{(k+1)})^\top-X^{(k)}(X^{(k)})^\top}_{\F}^{2}.
\end{align*} 
\end{lemma}
\begin{proof}
From equation \eqref{eqn:primal residual in dual}, we obtain
\begin{equation*}
R^{(k+1)}-R^{(k)}=\frac{1}{\tau\beta}\lrbracket{(\Lambda^{(k+1)}-\Lambda^{(k)})-(1-\delta)(\Lambda^{(k)}-\Lambda^{(k-1)})}.
\end{equation*}
Taking inner products with $2\tau\beta(\Lambda^{(k+1)}-\Lambda^{(k)})$ on both sides gives
\begin{align*}
&~2\tau\beta\inner{\Lambda^{(k+1)}-\Lambda^{(k)},R^{(k+1)}-R^{(k)}}\\
=&~2\inner{\Lambda^{(k+1)}-\Lambda^{(k)},(\Lambda^{(k+1)}-\Lambda^{(k)})-(1-\delta)(\Lambda^{(k)}-\Lambda^{(k-1)})}\\
=&~\norm{(\Lambda^{(k+1)}-\Lambda^{(k)})-(1-\delta)(\Lambda^{(k)}-\Lambda^{(k-1)})}_{\F}^{2}+\norm{\Lambda^{(k+1)}-\Lambda^{(k)}}_{\F}^{2}-(1-\delta)^2\norm{\Lambda^{(k)}-\Lambda^{(k-1)}}_{\F}^{2}\\
\geq &~\norm{\Lambda^{(k+1)}-\Lambda^{(k)}}_{\F}^{2}-(1-\delta)^2\norm{\Lambda^{(k)}-\Lambda^{(k-1)}}_{\F}^{2}.
\end{align*}
After rearranging, we arrive at
\begin{equation}\label{eqn:compensation inequality with inner product}
\begin{aligned}
&~\frac{1-\delta}{2\tau\beta}\norm{\Lambda^{(k+1)}-\Lambda^{(k)}}_{\F}^{2}\\
\leq&~\frac{(1-\delta)^2}{2\tau\beta}\norm{\Lambda^{(k)}-\Lambda^{(k-1)}}_{\F}^{2}+\inner{\Lambda^{(k+1)}-\Lambda^{(k)},R^{(k+1)}-R^{(k)}}-\frac{\delta}{2\tau\beta}\norm{\Lambda^{(k+1)}-\Lambda^{(k)}}_{\F}^{2}\\ 
\leq&~\frac{1-\delta}{2\tau\beta}\norm{\Lambda^{(k)}-\Lambda^{(k-1)}}_{\F}^{2}+\inner{\Lambda^{(k+1)}-\Lambda^{(k)},R^{(k+1)}-R^{(k)}}-\frac{\delta}{2\tau\beta}\norm{\Lambda^{(k+1)}-\Lambda^{(k)}}_{\F}^{2}.
\end{aligned}
\end{equation}
For arbitrary $\eps_X,\eps_Y>0$, it follows from Young's inequality that
\begin{equation*}
\inner{\Lambda^{(k+1)}-\Lambda^{(k)},Y^{(k+1)}-Y^{(k)}}\leq \frac{\eps_Y}{2}\norm{Y^{(k+1)}-Y^{(k)}}_{\F}^{2}+\frac{1}{2\eps_Y}\norm{\Lambda^{(k+1)}-\Lambda^{(k)}}_{\F}^{2}
\end{equation*}
and
\begin{align*}
&~-\inner{\Lambda^{(k+1)}-\Lambda^{(k)},X^{(k+1)}(X^{(k+1)})^\top-X^{(k)}(X^{(k)})^\top}\\
\leq&~\frac{\eps_X}{2}\norm{X^{(k+1)}(X^{(k+1)})^\top-X^{(k)}(X^{(k)})^\top}_{\F}^{2}+ \frac{1}{2\eps_X}\norm{\Lambda^{(k+1)}-\Lambda^{(k)}}_{\F}^{2}.
\end{align*}
Substituting the above inequalities into equation \eqref{eqn:compensation inequality with inner product}, we complete the proof.
\end{proof}

By invoking Lemma \ref{lemma:compensation inequality}, the positive term in Lemma \ref{lemma:warm-up Lyapunov} can be controlled after properly choosing parameters. Based on this, we construct the following surrogate sequence to monitor the behavior of the algorithm: 
\begin{equation}\label{eqn:Lyapunov function}
\Psi_k:=T_k+c\cdot\frac{1-\delta}{2\tau\beta}\norm{\Lambda^{(k)}-\Lambda^{(k-1)}}_{\F}^{2}+C_0,
\end{equation}
where $T_k$ is defined in equation \eqref{eqn:pre-surrogate sequence}, the constant $c$ is given in Assumption \ref{assumption}, and $C_0$ is defined as
\begin{equation*}
C_0:=\frac12\norm{\nabla f(0)}_{\F}^{2}+(L+1)p-f(0),
\end{equation*}
which ensures the lower boundedness of $\{\Psi_k\}$ (see Lemma \ref{lemma:Lyapunov lower bound}). The next lemma establishes a descent inequality for $\{\Psi_k\}$.
\begin{lemma}\label{lemma:Lyapunov decrease inequality}
Suppose that {\rm Assumption \ref{assumption}} holds. Then there exist constants $c_1,c_2,c_3>0$ such that 
\begin{equation}\label{eqn:Lyapunov decrease inequality}
\Psi_{k+1}-\Psi_k\leq -c_1\norm{Y^{(k+1)}-Y^{(k)}}_{\F}^{2}-c_2\norm{X^{(k+1)}(X^{(k+1)})^\top-X^{(k)}(X^{(k)})^\top}_{\F}^{2}-c_3\norm{\Lambda^{(k+1)}-\Lambda^{(k)}}_{\F}^{2}
\end{equation}
holds for all $k\geq 1$.
\end{lemma}
\begin{proof}
Combining Lemma \ref{lemma:warm-up Lyapunov} with Lemma \ref{lemma:compensation inequality}, we obtain
\begin{align*}
\Psi_{k+1}-\Psi_k\leq&~-\lrbracket{\frac{1}{\eta}-\frac{\beta+L}{2}-\frac{c\cdot\eps_{Y}}{2}}\norm{Y^{(k+1)}-Y^{(k)}}_{\F}^{2}\\
&~-\lrbracket{\frac{\beta\gamma}{n}-\frac{c\cdot\eps_{X}}{2}}\norm{X^{(k+1)}(X^{(k+1)})^\top-X^{(k)}(X^{(k)})^\top}_{\F}^{2}\\
&~-\lrbracket{c\bigg(\frac{\delta}{2\tau\beta}-\frac{1}{2\eps_{Y}}-\frac{1}{2\eps_{X}}\bigg)-\frac{(1-\delta)(1-\frac{\delta}{2})}{\tau\beta}}\norm{\Lambda^{(k+1)}-\Lambda^{(k)}}_{\F}^{2}.
\end{align*}
Next, we choose $\eps_X=\eps_Y=4\tau\beta/\delta>0$. Substituting this choice into the above estimate yields equation \eqref{eqn:Lyapunov decrease inequality} with 
$$c_1:=\frac{1}{\eta}-\frac{\beta+L}{2}-\frac{2c\tau\beta}{\delta},\quad c_2:=\lrbracket{\frac{\gamma}{n}-\frac{2c\tau}{\delta}}\beta,\quad c_3:=\frac{1}{\tau\beta}\lrbracket{\frac{c\delta}{4}-(1-\delta)\bigg(1-\frac{\delta}{2}\bigg)}.$$
The positivity of these constants follows from the condition $c>2(1-\delta)(2-\delta)/\delta$, together with the bound on $\tau$ in Assumption \ref{assumption}~\ref{asp:under relaxation param}.
\end{proof}

The next step is to show that $\{\Psi_k\}$ is bounded from below.
\begin{lemma}\label{lemma:Lyapunov lower bound}
Suppose that {\rm Assumption \ref{assumption}} holds. Then the sequence $\{\Psi_k\}$ is bounded from below.
\end{lemma}
\begin{proof}
We decompose $\Psi_k$ into two parts:
\begin{align*}
\Psi_k=&\underbrace{f(Y^{(k)})+\frac{\beta}{2}\norm{R^{(k)}}_{\F}^{2}+\frac{\delta(1-\delta)}{2\tau\beta}\norm{\Lambda^{(k)}}_{\F}^{2}+c\cdot\frac{1-\delta}{2\tau\beta}\norm{\Lambda^{(k)}-\Lambda^{(k-1)}}_{\F}^{2}+C_0}_{\text{Part 1}}\\
&+\underbrace{\inner{(1-\delta) \Lambda^{(k)}, R^k}-\frac{\delta(1-\delta)}{\tau \beta}
\norm{\Lambda^{(k)}}_{\F}^2}_{\text{Part 2}}.
\end{align*}

For Part 1, we first estimate the objective term. By Assumption \ref{assumption} \ref{asp:lipschitz smooth},
\begin{equation*}
f(Y^{(k)})\geq f(0)+\inner{\nabla f(0),Y^{(k)}}-\frac{L}{2}\norm{Y^{(k)}}_{\F}^{2}\geq -\frac{L+1}{2}\norm{Y^{(k)}}_{\F}^{2}-\frac12\norm{\nabla f(0)}_{\F}^{2}+f(0).
\end{equation*}
Since $R^{(k)}=Y^{(k)}-X^{(k)}(X^{(k)})^\top$ and $\norm{X^{(k)}(X^{(k)})^\top}_{\F}^{2}=p$, we have $\norm{Y^{(k)}}_{\F}^{2}\leq 2(\norm{R^{(k)}}_{\F}^{2}+p)$. Therefore, by Assumption \ref{assumption} \ref{asp:penalty param},
\begin{equation*}
\text{Part 1}\geq f(Y^{(k)})+\frac{\beta}{2}\norm{R^{(k)}}_{\F}^{2}+C_0\geq \lrbracket{\frac{\beta}{2}-L-1}\norm{R^{(k)}}_{\F}^{2}\geq 0.
\end{equation*}

For Part 2, it follows from direct calculations that
\begin{align*}
\text{Part 2}& = \inner{(1-\delta)\Lambda^{(k)},\frac{\Lambda^{(k)}-(1-\delta)\Lambda^{(k-1)}}{\tau\beta}-\frac{\delta}{\tau\beta}\Lambda^{(k)}}\\
& = \frac{(1-\delta)^2}{2\tau\beta}\lrbracket{\norm{\Lambda^{(k)}}_{\F}^{2}-\norm{\Lambda^{(k-1)}}_{\F}^{2}+\norm{\Lambda^{(k)}-\Lambda^{(k-1)}}_{\F}^{2}}\\
& \geq  \frac{(1-\delta)^2}{2\tau\beta}\lrbracket{\norm{\Lambda^{(k)}}_{\F}^{2}-\norm{\Lambda^{(k-1)}}_{\F}^{2}}.
\end{align*}

Combining the estimates of Part 1 and Part 2, we obtain
\begin{equation*}
\Psi_k\geq \frac{(1-\delta)^2}{2\tau\beta}\lrbracket{\norm{\Lambda^{(k)}}_{\F}^{2}-\norm{\Lambda^{(k-1)}}_{\F}^{2}}.
\end{equation*}
For any $K\geq 1$, summing the above inequality from $k=1$ to $K$ yields
\begin{equation*}
\sum_{k=1}^K\Psi_k\geq \frac{(1-\delta)^2}{2\tau\beta}\lrbracket{\norm{\Lambda^{(K)}}_{\F}^{2}-\norm{\Lambda^{(0)}}_{\F}^{2}}\geq -\frac{(1-\delta)^2}{2\tau\beta}\norm{\Lambda^{(0)}}_{\F}^{2}.
\end{equation*}
Hence, $\{\sum_{k=1}^K\Psi_k\}$ is bounded below.
Together with the fact that $\{\Psi_k\}$ is nonincreasing from Lemma \ref{lemma:Lyapunov decrease inequality}, we conclude that $\Psi_k\geq 0$ for all $k\geq 1$.
\end{proof}

An immediate consequence of Lemmas \ref{lemma:Lyapunov decrease inequality} and \ref{lemma:Lyapunov lower bound} is as follows.
\begin{corollary}\label{corollary:summability}
Suppose that {\rm Assumption \ref{assumption}} holds. Then
\begin{equation*}
\sum_{k=1}^{\infty}\norm{Y^{(k+1)}-Y^{(k)}}_{\F}^{2},~~\sum_{k=1}^{\infty}\norm{X^{(k+1)}(X^{(k+1)})^\top-X^{(k)}(X^{(k)})^\top}_{\F}^{2},~~\sum_{k=1}^{\infty}\norm{\Lambda^{(k+1)}-\Lambda^{(k)}}_{\F}^{2}
\end{equation*}
are all finite. In particular,  $Y^{(k+1)}-Y^{(k)}\to 0,~X^{(k+1)}(X^{(k+1)})^\top-X^{(k)}(X^{(k)})^\top\to 0,~\Lambda^{(k+1)}-\Lambda^{(k)}\to 0$. Moreover, we have $G^{(k)}\to 0$.
\end{corollary}

The remaining issue is to establish the boundedness of iterates.
\begin{lemma}\label{lemma:boundedness of iterates}
Under {\rm Assumption \ref{assumption}}, the sequences $\{Y^{(k)}\}$ and $\{\Lambda^{(k)}\}$ generated by Algorithm \ref{algn:DP-SSA} are bounded. 
\end{lemma}
\begin{proof}
In the proof of Lemma \ref{lemma:Lyapunov lower bound}, we have established that
\begin{equation*}
\Psi_k\geq  \frac{\delta(1-\delta)}{2\tau\beta}\norm{\Lambda^{(k)}}_{\F}^{2}+ \frac{(1-\delta)^2}{2\tau\beta}\lrbracket{\norm{\Lambda^{(k)}}_{\F}^{2}-\norm{\Lambda^{(k-1)}}_{\F}^{2}}. 
\end{equation*}
Since $\{\Psi_k\}$ is nonincreasing, we have $\Psi_k\leq \Psi_1$, which further implies
\begin{equation*}
\norm{\Lambda^{(k)}}_{\F}^{2}\leq (1-\delta)\norm{\Lambda^{(k-1)}}_{\F}^{2}+\frac{2\tau\beta}{1-\delta}\Psi_1.
\end{equation*}
Applying this inequality recursively gives
\begin{equation*}
\norm{\Lambda^{(k)}}_{\F}^{2}\leq (1-\delta)^k\norm{\Lambda^{(0)}}_{\F}^{2}+\frac{2\tau\beta}{1-\delta}\Psi_1 \sum_{\ell=0}^{k-1}(1-\delta)^{\ell}\leq \norm{\Lambda^{(0)}}_{\F}^{2}+\frac{2\tau\beta}{\delta(1-\delta)}\Psi_1.
\end{equation*}
Thus, $\{\Lambda^{(k)}\}$ is bounded.
Finally, we rewrite the dual update \eqref{eqn:damped dual update} as
\begin{equation*}
Y^{(k)}=X^{(k)}(X^{(k)})^\top+\frac{1}{\tau\beta}\lrbracket{\Lambda^{(k)}-(1-\delta)\Lambda^{(k-1)}}.
\end{equation*}
Since $\norm{X^{(k)}(X^{(k)})^\top}_{\F}=\sqrt{p}$ and $\{\Lambda^{(k)}\}$ is bounded, the sequence $\{Y^{(k)}\}$ is bounded  as well.
\end{proof}

We now state the main convergence result for Algorithm \ref{algn:DP-SSA}.
\begin{theorem}\label{thm:main theorem}
Suppose that {\rm Assumption \ref{assumption}} holds. Let $\{(X^{(k)},Y^{(k)},\Lambda^{(k)})\}$ be the sequence generated by Algorithm \ref{algn:DP-SSA}. Then $\{(X^{(k)},Y^{(k)},\Lambda^{(k)})\}$ has at least one accumulation point, and for each accumulation point $(X^\star,Y^\star,\Lambda^\star)$, there exist multipliers $\Sigma^\star\in\S^p$ and $\bmu^\star=(\mu_1^\star,\ldots,\mu_m^\star)^\top\in\R^m$ such that 
\begin{equation*}
\norm{F(X^\star,\bmu^\star)X^{\star}-X^{\star}\Sigma^{\star}}_{\F}\leq\hat C_1\delta,\qquad \max_{1\leq j\leq m}\abs{\trace{(X^{\star\top}W_jX^{\star})}-b_j}\leq \hat C_2\delta,
\end{equation*}
where 
$$\hat C_1:=\frac{L\sqrt{p(n-p)}}{\tau\beta}\norm{\Lambda^{\star}}_{\F},\quad \hat C_2:=\frac{1}{\tau\beta}\norm{\Lambda^{\star}}_{\F}\cdot\max_{1\leq j\leq m}\norm{W_j}_{\F}.$$

\end{theorem}

\begin{proof}
The existence of an accumulation point follows directly from Lemma \ref{lemma:boundedness of iterates} and the compactness of the Stiefel manifold. By equation \eqref{eqn:projected gradient} and the characterization of the orthogonal projection onto $\calT_{\calY}$, there exists a unique matrix $N^{(k)}\in \Span\{W_1,\ldots,W_m\}$ such that
\begin{equation}\label{eqn:projected gradient in main theorem}
G^{(k)}=\nabla f(Y^{(k)})+(1-\delta)\Lambda^{(k)}+\beta(Y^{(k)}-X^{(k+1)}(X^{(k+1)})^\top)+N^{(k)}.
\end{equation}
Since $\norm{G^{(k)}}_{\F}\to 0$ , $\{X^{(k)}\}, \{Y^{(k)}\}$ and $\{\Lambda^{(k)}\}$ are bounded, we have that $\{N^{(k)}\}$ is bounded as well. We further define 
\begin{equation}\label{eqn:Lagrange multiplier}
\Sigma^{(k)}:=(X^{(k)})^\top\Big(\nabla f(X^{(k)}(X^{(k)})^\top)+N^{(k)}\Big)X^{(k)}\in \S^p.    
\end{equation}
Passing to a convergent subsequence without relabeling, we may assume $X^{(k)}\to X^\star\in\St_p(\R^n)$, $Y^{(k)}\to Y^\star\in\calY$, $\Lambda^{(k)}\to\Lambda^\star$, $\Sigma^{(k)}\to\Sigma^\star$, and $N^{(k)}\to N^{\star}=\sum_{j=1}^{m}\mu_{j}^{\star}W_j\in \S^n$ for some $\bm{\mu}^{\star}\in \R^{m}$. 

\par It follows from equations \eqref{eqn:projected gradient in main theorem} and \eqref{eqn:Lagrange multiplier} that
\begin{align}
&~\Big(\nabla f(X^{(k)}(X^{(k)})^\top)+N^{(k)}\Big)X^{(k)}-X^{(k)}\Sigma^{(k)}\nonumber\\
=&~ {\rm P}_{X^{(k)}}^\perp \Big(\nabla f(X^{(k)}(X^{(k)})^{\top})+N^{(k)}\Big)X^{(k)}\nonumber\\
= &~{\rm P}_{X^{(k)}}^\perp\Big(-\beta(Y^{(k)}-X^{(k+1)}(X^{(k+1)})^\top)-(1-\delta)\Lambda^{(k)}\Big)X^{(k)}\nonumber\\
&~+{\rm P}_{X^{(k)}}^\perp\Big(G^{(k)}+\nabla f(X^{(k)}(X^{(k)})^\top)-\nabla f(Y^{(k)})\Big)X^{(k)}.\label{eqn:substationarity}
\end{align}
Note that the first term on the right-hand side vanishes as $k\to\infty$ by Corollary \ref{corollary:summability} and 
\begin{equation*}
\projj\Big(\beta(Y^{(k)}-X^{(k+1)}(X^{(k+1)})^\top)+(1-\delta)\Lambda^{(k)}\Big)X^{(k+1)}=0,
\end{equation*}
which follows from the first-order optimality condition of the $X$-subproblem \eqref{eqn:X subproblem}. For the second term, by Assumption \ref{assumption}~\ref{asp:lipschitz smooth}, we have
\begin{align*}
\norm{{\rm P}_{X^{(k)}}^\perp\Big(G^{(k)}+\nabla f(X^{(k)}(X^{(k)})^\top)-\nabla f(Y^{(k)})\Big)X^{(k)}}_{\F}\leq \sqrt{p(n-p)}\Big(\norm{G^{(k)}}_{\F}+L\norm{R^{(k)}}_{\F}\Big).
\end{align*}
Taking the limit $k\to\infty$ in equation  \eqref{eqn:substationarity} and using both Corollary \ref{corollary:summability} and equation \eqref{eqn:primal residual in dual}, we arrive at
\begin{equation*}
\norm{F(X^\star,\bmu^\star)X^\star-X^\star\Sigma^\star}=\norm{\Big(\nabla f(X^{\star}X^{\star^\top})+N^{\star}\Big)X^{\star}-X^{\star}\Sigma^{\star}}_{\F}\leq \frac{L\sqrt{p(n-p)}}{\tau\beta}\norm{\Lambda^{\star}}_{\F}\delta.
\end{equation*}
For feasibility, since $Y^{\star}\in \calY$, we have $\trace{(W_jY^{\star})}=b_j$ for $j=1,\ldots,m$. Hence
\begin{equation*}
\abs{\trace{(X^{\star \top}W_jX^{\star})}-b_j}=\abs{\trace{(W_j(X^{\star}X^{\star^{\top}}-Y^{\star}))}}\leq \norm{W_j}_{\F}\norm{R^{\star}}_{\F}=\frac{\delta}{\tau\beta}\norm{\Lambda^{\star}}_{\F}\norm{W_j}_{\F}.
\end{equation*}
Taking the maximum over $j$ gives the feasibility bound and completes the proof.
\end{proof}

\section{Numerical experiments}\label{sec:numerical experiments}

\par In this section, we test the proposed DASSP on synthetic and realistic CDFT problems. In Section \ref{subsec:default settings}, we determine a default setting of parameters in DASSP by experiments on synthetic test instances. In Sections \ref{subsec:inexactness of quadratic penalty}-\ref{subsec:charge-constrained DFT}, we demonstrate the effectiveness, efficiency, and robustness of DASSP on charge-transfer and charge-localization problems, through numerical comparison with two classes of algorithms mentioned in Section \ref{sec:introduction}. 


\medskip

\par\noindent\textbf{Algorithms in comparison.} 
The first class is the quadratic penalty method \cite{kresse1996efficiency,kresse1996efficient,ma2015constrained}. 
As discussed in Section \ref{sec:introduction}, the penalty problem \eqref{eqn:CDFT quadratic penalty} can be solved by either the SCF iteration or optimization algorithms on the Stiefel manifold. The former option is abbreviated as QP-SCF and is implemented using the PySCF package \cite{sun2020recent}. For the latter, we invoke the PCAL algorithm on the Stiefel manifold with its default setting \cite{gao2022ofdf,gao2019plam} and call the resulting method QP-PCAL. The second class is the double-loop method proposed by Wu and Van Voorhis \cite{wu2005direct}, referred to as the WV method. Its implementation follows from the CDFT example in the PySCF package\footnote{Available at \href{https://github.com/pyscf/pyscf/blob/master/examples/1-advanced/033-constrained_dft.py}{https://github.com/pyscf/pyscf/blob/master/examples/1-advanced/033-constrained\_dft.py}.}. We further employ the post-DIIS method \cite{pulay1980convergence,pulay1982diis} with a subspace dimension of eight in both the QP-SCF and WV methods for acceleration. 

\medskip

\par\noindent\textbf{Stopping criteria.} We define the following two residuals: for any $X\in\R^{n\times p}$ and $\bmu\in\R^m$,
$$r_1(X,\bmu):=\norm{(I-XX^\top)F(X,\bmu)X}_{\rm F},\quad r_2(X):=\max_{1\le j\le m}\abs{\trace\big(X^\top W_jX\big)-b_j}.$$
The former denotes the stationarity (resp., substationarity) violation for the QP-SCF and QP-PCAL (resp., WV and DASSP) methods (cf. equation \eqref{eqn:CDFT stationarity conditions nepv} with $\Sigma:=X^\top F(X,\bmu)X$), while the latter  represents the feasibility violation for the WV and DASSP methods. Unless otherwise specified, 
the QP-SCF and QP-PCAL are declared converged when $r_1\big(X^{(k)},\tilde{\bmu}^{(k)}\big)$ falls below $10^{-8}$, where $\tilde{\bmu}^{(k)}:=(\tilde\mu_1^{(k)},\ldots,\tilde\mu_m^{(k)})^\top\in\R^m$ and $\tilde\mu_j^{(k)}:=2\sigma(\trace((X^{(k)})^\top W_jX^{(k)})-b_j)$ ($j=1,\ldots,m$). 
Moreover, an orthogonalization post-processing step is enabled for QP-PCAL when the final manifold constraint violation exceeds $10^{-13}$ \cite{gao2022ofdf}. In the WV method, the inner Newton step is terminated when $r_2(X^{(k)})$ falls below $10^{-8}$ or the inner iteration number exceeds a prespecified value; at most $20$ inner Newton steps are allowed during the first ten outer iterations, and at most $50$ inner steps are allowed thereafter. The outer iteration is terminated when $r_1\big(X^{(k)},\bmu^{(k)}\big)$ is smaller than $10^{-8}$; see Section \ref{sec:introduction} for the definition of $\bmu^{(k)}$. DASSP is terminated when $r_1\big(X^{(k)},\hat{\bmu}^{(k)}\big)$ and $r_2(X^{(k)})$ are both below $10^{-8}$, where $\hat{\bmu}^{(k)}\in\R^m$ is determined through equation \eqref{eqn:projection linear system}. We also set maximum numbers of iterations for the algorithms, which will be specified later. 

\medskip

\par\noindent\textbf{Testing molecules and environment.} All the test molecules in this section are described by restricted or unrestricted Kohn-Sham DFT (see Appendix \ref{appsec:UKS functional}); basis sets and exchange-correlation functionals in use will be detailed later. The molecular geometries used in Sections \ref{subsec:default settings}-\ref{subsec:efficiency test} can be found in the Computational Chemistry Comparison and Benchmark DataBase (CCCBDB) \cite{johnson2022nist}, whereas the geometry used in Section \ref{subsec:charge-constrained DFT} is from \cite{ahart2022cp2k}. All the experiments are performed on a workstation with two Intel(R) Xeon(R) Processors Gold 5317 (at 3.00GHz$\times$12, 18M Cache) and 512GB of RAM under Ubuntu 20.04.6. The codes are implemented in Python 3.10.19 with PySCF 2.11.0 \cite{sun2020recent}. 

\subsection{Parameter settings}\label{subsec:default settings}
\par In this subsection, we determine the default parameters for the proposed DASSP. We test DASSP on three molecules: \ce{CO2}, \ce{C2H6}, and \ce{C6H5OCH3}, treated within restricted Kohn-Sham DFT in PySCF. All calculations employ the {\tt sto-3g} atomic-orbital basis set \cite{hehre1969self} and the PBE exchange-correlation functional \cite{perdew1996pbe}; as such, we have $(n,p)=(15,11)$ for \ce{CO2}, $(n,p)=(16,9)$ for \ce{C2H6}, and $(n,p)=(48,29)$ for \ce{C6H5OCH3}. We impose $m=3, 6, 9$ quadratic constraints on \ce{CO2}, \ce{C2H6}, and \ce{C6H5OCH3}, respectively. For each molecule, the symmetric constraint matrices $\{W_j\}_{j=1}^{m}$ are generated independently from the Gaussian distribution, and the corresponding $\{b_j\}_{j=1}^{m}$ are defined using a randomly generated reference point to ensure feasibility. 

\par The algorithm is initialized randomly in this subsection. Specifically, $X^{(0)}\in\St_p(\R^n)$ is generated as a random matrix from the Gaussian distribution followed by a QR decomposition, $Y^{(0)}$ is initialized as the projection of $X^{(0)}(X^{(0)})^{\top}$ onto the affine set $\calY$, and the initial multiplier is $\Lambda^{(0)}=0$. For a fair comparison, the same initialization is used within each parameter sweep. 

\par We tune four components of DASSP in the following order: the stepsize strategy in the $Y$-update \eqref{eqn:Y update}, the damping parameter $\delta$, the penalty parameter $\beta$,  and the under-relaxation parameter $\tau$. In each experiment, only one parameter is varied while keeping the others fixed.

    
    
    
    

\begin{figure}[!t]
    \centering
    \includegraphics[width=0.95\linewidth]{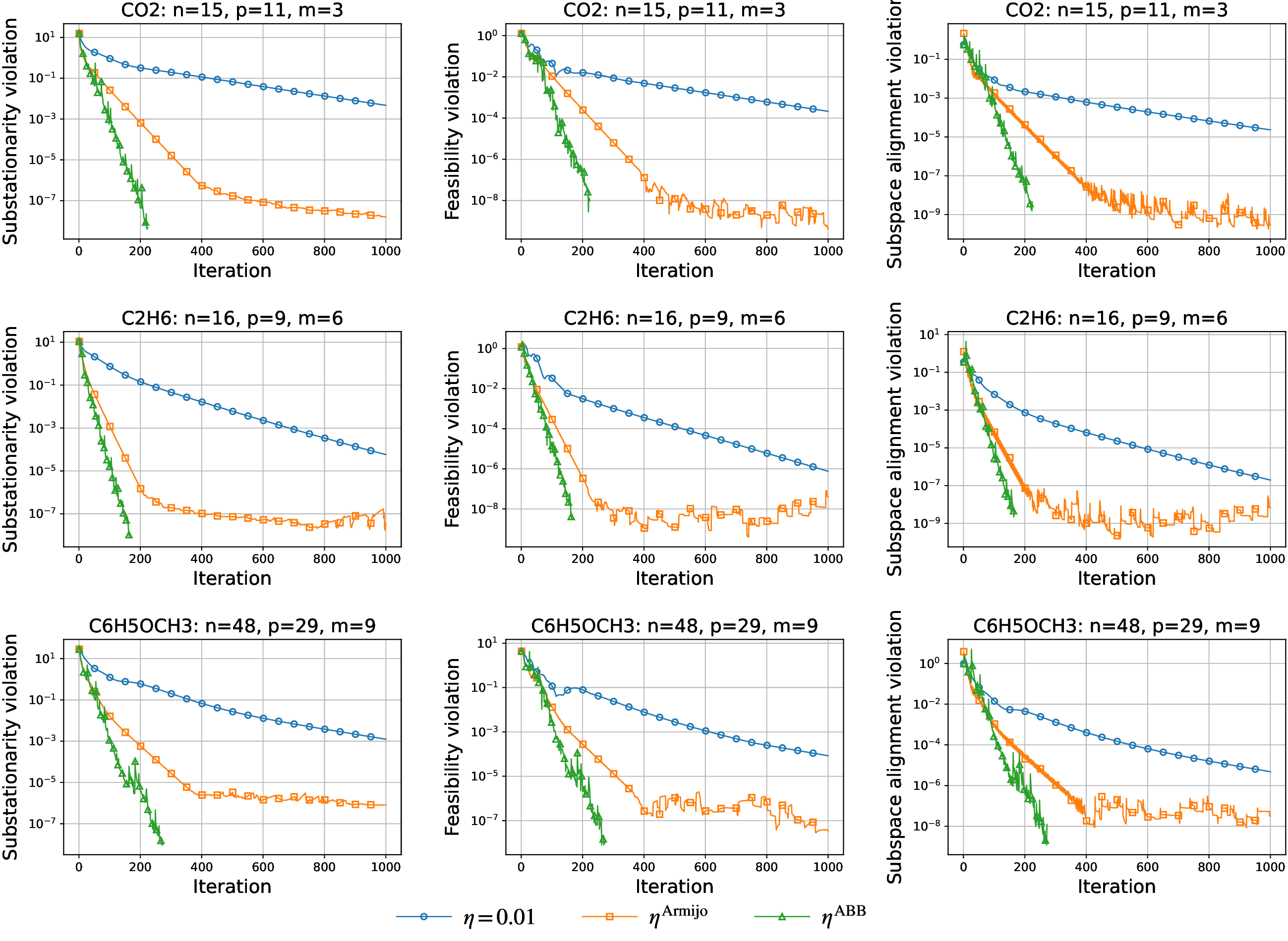}
    \caption{Comparison between different stepsize strategies. The blue curves with circle markers, orange curves with square markers, and green curves with triangle markers represent the results using the constant, Armijo backtracking, and clipped ABB strategies, respectively. From top to bottom: \ce{CO2}, \ce{C2H6}, and \ce{C6H5OCH3}. From left to right: substationarity violation, feasibility violation, and subspace alignment violation.}
    \label{fig:stepsize comparison}
\end{figure}

\medskip

\par\noindent\textbf{Stepsize strategy.} We fix $\beta=20,\tau=0.2$, and $\delta=10^{-10}$ and compare three strategies for the projected gradient stepsize $\eta_k$ in the $Y$-update \eqref{eqn:Y update}: (1) constant stepsize\footnote{We note that this stepsize has been slightly tuned to improve the efficiency of the constant-stepsize strategy}.: $\eta_k\equiv10^{-2}$; (2) Armijo backtracking line search: starting from $\eta_{k}=10^{-2}$, the trial stepsize is halved until
\begin{equation*}
\calL_{\beta,\delta}(X^{(k+1)},Y^{(k)}-\eta_k G^{(k)},\Lambda^{(k)})\leq \calL_{\beta,\delta}(X^{(k+1)},Y^{(k)},\Lambda^{(k)})-10^{-3}\cdot\eta_k\norm{G^{(k)}}_{\F}^{2}
\end{equation*}
is satisfied;
(3) clipped alternating Barzilai-Borwein (ABB) stepsize \cite{dai2005projected}: 
\begin{equation*}
\tilde\eta_{k}^{\text{ABB}}:=
\begin{cases}
10^{-2}, & \text{if}~k=0,\\
\frac{\inner{S^{(k-1)},S^{(k-1)}}}{|\inner{S^{(k-1)},Z^{(k-1)}}|}, & \text{if}~k\ne0~\text{is~odd},\\
\frac{|\inner{S^{(k-1)},Z^{(k-1)}}|}{\inner{Z^{(k-1)},Z^{(k-1)}}}, & \text{if}~k\ne0~\text{is~even},
\end{cases}\qquad 
\eta_k^{\text{ABB}}:=\min\{10,\max\{10^{-6},\tilde\eta_k^{\text{ABB}}\}\},
\end{equation*}
where $S^{(k-1)}:=Y^{(k)}-Y^{(k-1)}\in \S^n$ and $Z^{(k-1)}:=G^{(k)}-G^{(k-1)}\in \S^n$. 

\par Figure \ref{fig:stepsize comparison} shows the results of DASSP with different strategies for $\eta_k$, where we also include the evolution of the subspace alignment violation, defined as $\snorm{R^{(k)}}_{\rm F}$ (see equation \eqref{eqn:subspace alignment violation} and same below). Among the three strategies, the clipped ABB gives the most favorable overall performance. We therefore adopt $\eta_k=\eta_k^{\text{ABB}}$ in all subsequent experiments.

\medskip

\par \noindent \textbf{Damping parameter $\delta$.} We fix $\eta_k=\eta_k^{\text{ABB}}, \beta=20$, and $\tau=0.2$, and vary the damping parameter over $\{10^{-2},10^{-4},10^{-6},10^{-8},10^{-10},0\}$. The numerical results are presented in Figure \ref{fig:delta comparison}. 
For $\delta>0$, we observe that DASSP produces an approximate KKT point and the accuracy is improved as $\delta$ decreases. This is consistent with Theorem \ref{thm:main theorem}. In particular, $\delta=10^{-10}$ enables the algorithm to satisfy the stopping criterion in all tests. We also find that the choice $\delta=0$, i.e., dual update without damping, leads to convergence across the board. Since the convergence curves of $\delta=10^{-10}$ and $\delta=0$ are nearly indistinguishable in our tests and using full dual updates is potentially more advantageous for efficiency, we adopt the choice $\delta=0$ for ensuing experiments, even though its convergence is not covered in this work. 

    
    
    
    


\begin{figure}[!t]
    \centering
    \includegraphics[width=0.95\linewidth]{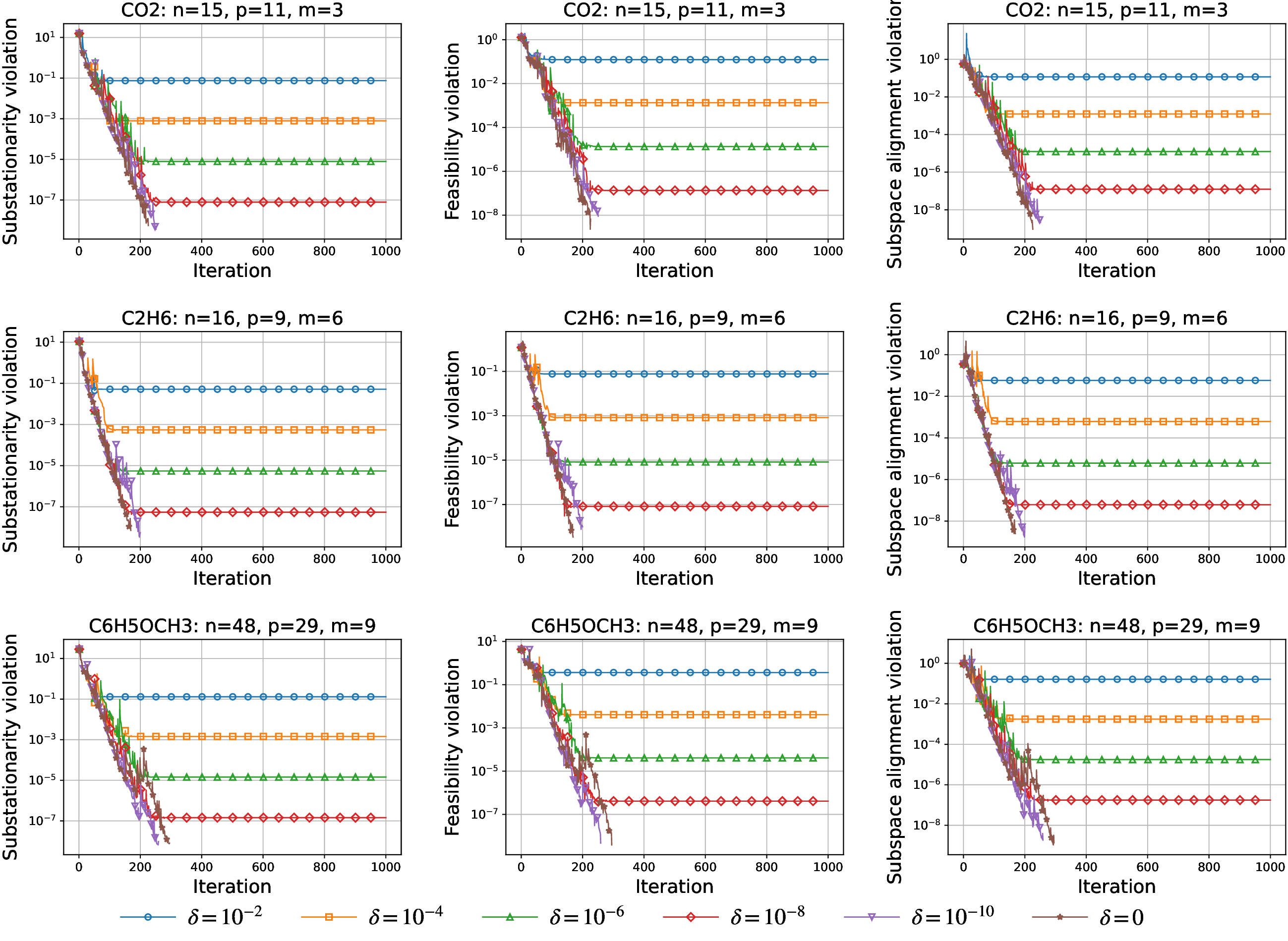}
    \caption{Comparison between different values of $\delta$. The blue curves with circle markers, orange curves with square markers, green curves with triangle markers, red curves with diamond markers, purple curves with down-pointing triangle markers, and brown curves with star markers represent the results using $\delta=10^{-2},10^{-4},10^{-6},10^{-8},10^{-10},0$, respectively. From top to bottom: \ce{CO2}, \ce{C2H6}, and \ce{C6H5OCH3}. From left to right: substationarity violation, feasibility violation, and subspace alignment violation.}
    \label{fig:delta comparison}
\end{figure}

\medskip

\par\noindent\textbf{Penalty parameter $\beta$.} We fix $\eta_k=\eta_k^{\text{ABB}}$, $\delta=0$, and $\tau=0.2$ and vary the penalty parameter $\beta$ within $\{5,10,20,30,40,50,100\}$. The numerical results are collected in Figure \ref{fig:beta comparison}. We observe that DASSP manages to converge even when $\beta$ is chosen as small as five, while its convergence behavior is strongly dependent on the value of $\beta$. Across the considered tests, the choice $\beta=20$ achieves the best overall performance. The efficiency is significantly hindered if $\beta$ is chosen to be either too small or too large relative to 20. 
Subsequently, we choose $\beta=20$ as the default  value.

    
    
    
    


\begin{figure}[!t]
    \centering
    \includegraphics[width=0.95\linewidth]{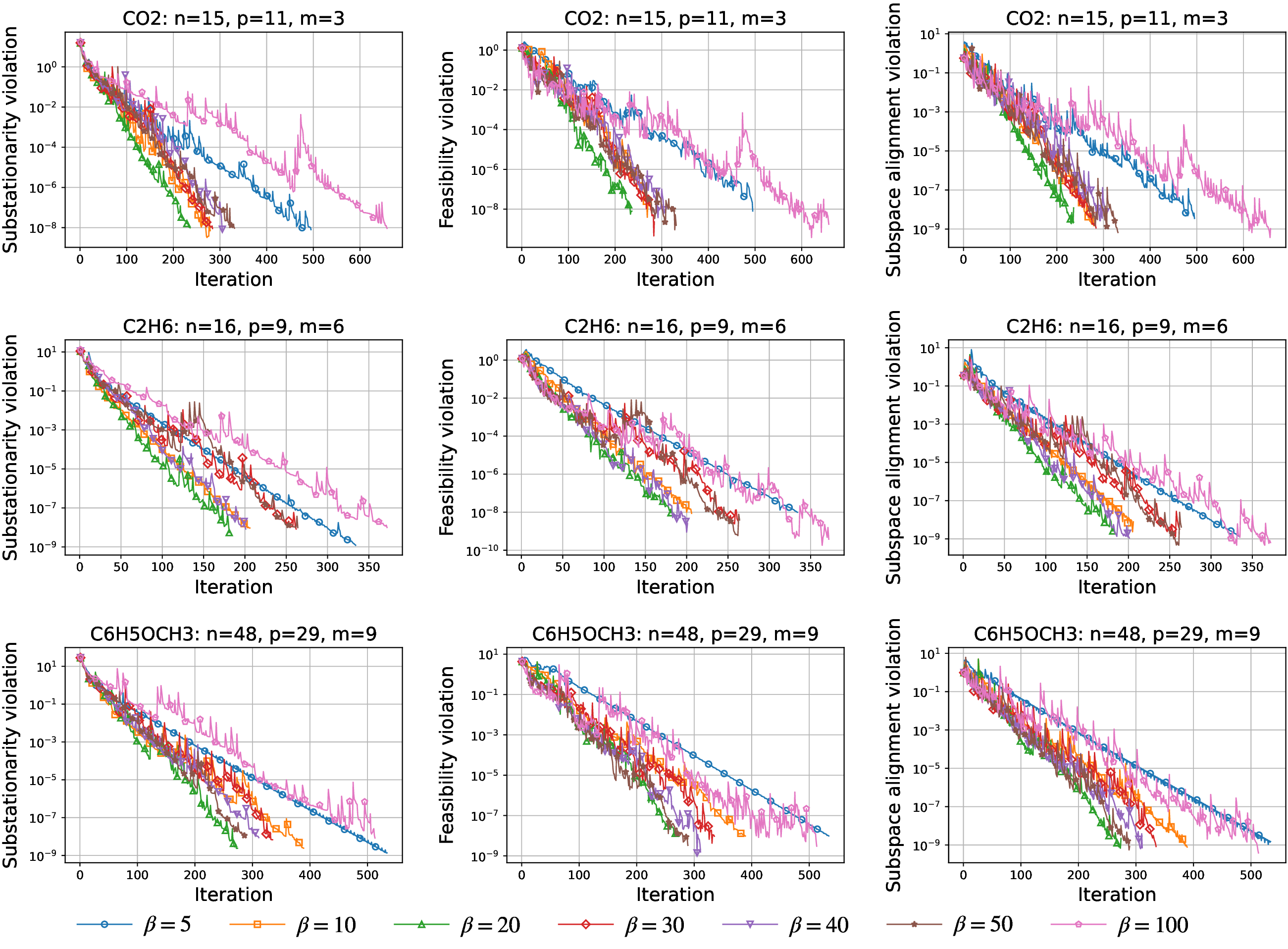}
    \caption{Comparison between different values of $\beta$. The blue curves with circle markers, orange curves with square markers, green curves with triangle markers, red curves with diamond markers, purple curves with down-pointing triangle markers, brown curves with star markers, and pink curves with pentagon markers represent the results using $\beta=5,10,20,30,40,50,100$, respectively. From top to bottom: \ce{CO2}, \ce{C2H6}, and \ce{C6H5OCH3}. From left to right: substationarity violation, feasibility violation, and subspace alignment violation.}
    \label{fig:beta comparison}
\end{figure}

\medskip

\par \noindent \textbf{Under-relaxation parameter $\tau$.} We fix $\eta_k=\eta_k^{\text{ABB}}, \beta=20$ and $\delta=0$ and vary $\tau$ within $\{0.1,0.2,\ldots,0.8\}$. The numerical results are reported in Figure \ref{fig:tau comparison}. We clearly observe that values in the range $[0.1,0.5]$ lead to much steadier convergence, with $\tau=0.2$ and $\tau=0.3$ the best two, in comparison with larger values. This indicates that the algorithm requires sufficient under-relaxation in dual updates to handle the nonlinear subspace alignment constraint. In the following, we set $\tau = 0.2$ as default. 


    
    
    
    


\begin{figure}[!t]
    \centering
    \includegraphics[width=0.95\linewidth]{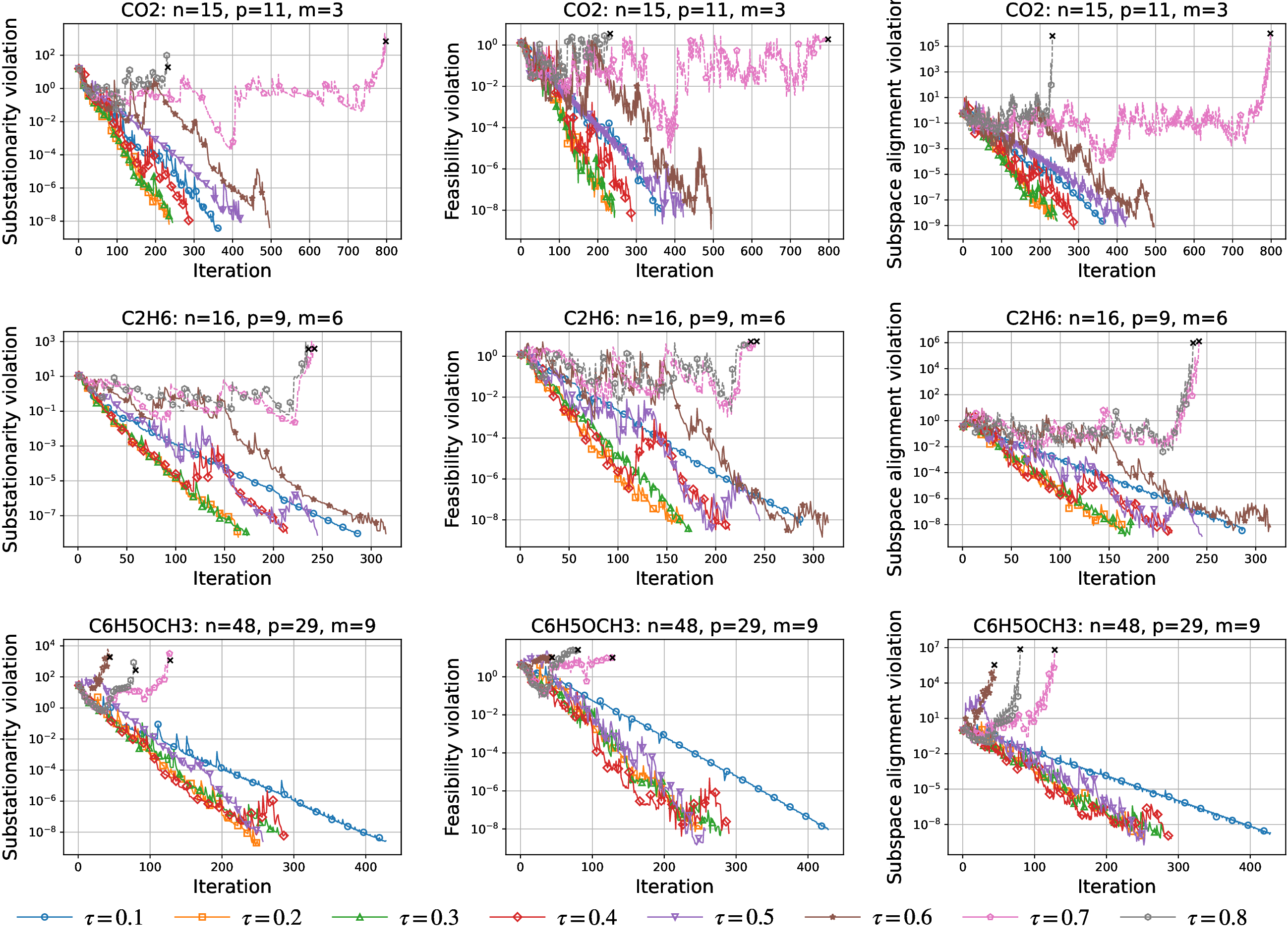}
    \caption{Comparison between different values of $\tau$. The blue curves with circle markers, orange curves with square markers, green curves with triangle markers, red curves with diamond markers, purple curves with down-pointing triangle markers, brown curves with star markers, pink curves with pentagon markers, and gray curves with hexagon markers represent the results using $\tau=0.1,0.2,\ldots,0.8$, respectively. From top to bottom: \ce{CO2}, \ce{C2H6}, and \ce{C6H5OCH3}. From left to right: substationarity violation, feasibility violation, and subspace alignment violation.}
    \label{fig:tau comparison}
\end{figure}

\medskip

\par These results lead to our default setting $(\eta_k,\beta,\delta,\tau)=(\eta_k^{\text{ABB}},20,0,0.2)$ in DASSP. These default parameters will be used in the remaining experiments unless otherwise stated. We remark that these choices are not guaranteed to be optimal; the parameter tuning is not comprehensive in terms of both testing instances and tuning strategy. An adaptive way to choose these parameters would be preferable and will be studied in the future.  

\subsection{Effectiveness in handling nonconvex quadratic constraints}\label{subsec:inexactness of quadratic penalty}

\par In this subsection, we demonstrate the effectiveness of DASSP in dealing with the nonconvex quadratic constraints through the numerical comparison with the QP-SCF, QP-PCAL, and WV methods. The experiments are conducted on an \ce{HCN} molecule under a charge-transfer process. We describe the molecule by restricted Kohn-Sham DFT, adopting the {\tt sto-3g} atomic-orbital basis set \cite{hehre1969self} and the PBE exchange-correlation functional \cite{perdew1996pbe}; therefore, we have $(n,p)=(22,7)$. The charge-transfer process is imposed on the localized populations of the \ce{C} $2p$ and \ce{N} $2p$ subspaces with the shift amount $s=0.5$; hence, $m=2$. The construction of nonconvex quadratic constraints associated with the charge transfer is detailed in Appendix \ref{appsec:nonconvex quadratic constraints}. All methods are initialized from the {\texttt{minao}} guess provided by PySCF \cite{almlof1982principles,van2006starting}; for DASSP, we further set $\Lambda^{(0)}=0$. The penalty parameter $\sigma$ in the QP-SCF and QP-PCAL is varied within $\{1,10,100,1000\}$. The maximum (outer-loop) iteration number for both QP-SCF and WV methods is set to 500. 



\par 
The numerical results are reported in Table \ref{tab:hcn_shift05_comparison}. 
\begin{table}[!t]
\centering
\caption{Comparison between QP-SCF, QP-PCAL, WV, and DASSP on the HCN charge-transfer test. The final energies, feasibility violations (Feas.), numbers of iterations (Iter.) and wall-clock times are listed. For the WV method, the values outside and inside the parentheses refer to the outer iteration numbers and the averaged inner iteration numbers, respectively.}
\label{tab:hcn_shift05_comparison}
\begin{tabular}{l|cccr}
\hline\hline
\multicolumn{1}{c|}{Algorithm} & Energy ($E_h$) & Iter.  & Feas. & Time (s) \\
\hline
QP-SCF ($\sigma=1$)
& $-9.203372001{\rm e}{+01}$ & $71$ & $8.286{\rm e}{-02}$ & $3.14$ \\

QP-SCF ($\sigma=10$)
& $-8.172849033{\rm e}{+01}$ & $500$ & $2.917{\rm e}{+00}$ & $21.24$ \\

QP-SCF ($\sigma=100$)
& $-8.627036448{\rm e}{+01}$ & $500$ & $2.917{\rm e}{+00}$ & $20.98$ \\

QP-SCF ($\sigma=1000$)
& $-8.213467723{\rm e}{+01}$ & $500$ & $2.917{\rm e}{+00}$ & $29.00$ \\
\hline
QP-PCAL ($\sigma=1$)
& $-9.203372001{\rm e}{+01}$ & $199$ & $8.286{\rm e}{-02}$ & $2.84$ \\

QP-PCAL ($\sigma=10$)
& $-9.201331705{\rm e}{+01}$ & $302$ & $1.052{\rm e}{-02}$ & $4.30$ \\

QP-PCAL ($\sigma=100$)
& $-9.201061596{\rm e}{+01}$ & $805$ & $1.087{\rm e}{-03}$ & $11.42$ \\

QP-PCAL ($\sigma=1000$)
& $-9.201033596{\rm e}{+01}$ & $2085$ & $1.091{\rm e}{-04}$ & $29.58$ \\
\hline
WV
& $-9.201030473{\rm e}{+01}$ & $85(8.53)$ & $7.303{\rm e}{-09}$ & $13.68$ \\
\hline
DASSP
& $-9.201030473{\rm e}{+01}$ & $293$ & $3.685{\rm e}{-09}$ & $3.08$ \\
\hline\hline
\end{tabular}
\end{table}
The QP-PCAL (with any $\sigma$), WV, and DASSP are all found to converge, while the QP-SCF only works with $\sigma=1$. When using a larger $\sigma$, the QP-SCF suffers from severe oscillations in the SCF iteration, as shown by the energy curves in Figure \ref{fig:QPSCF_EDFT_mu_curves}. 
\begin{figure}[htb]
    \centering
    \includegraphics[width=1\linewidth]{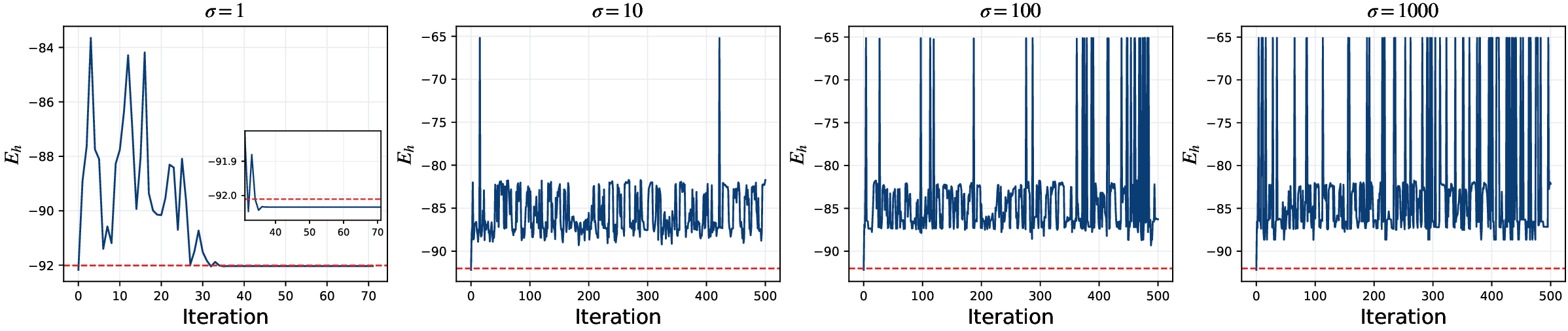}
    \caption{Energy curves of the QP-SCF with different values of $\sigma$. The red dashed line represents the energy obtained by DASSP.}
    \label{fig:QPSCF_EDFT_mu_curves}
\end{figure}
This instability can be understood from the fixed-point viewpoint of the SCF iteration. As shown in \cite{liu2015analysis}, the Jacobian of the SCF fixed-point map depends on the derivative of the effective Hamiltonian. 
For the QP-SCF, such a derivative contains an additional term proportional to $\sigma$. A large penalty parameter $\sigma$ could then enlarge the norm of the Jacobian of the SCF map, thereby spoiling the local contractivity of the SCF iteration. 


\par In contrast to QP-SCF, the QP-PCAL converges robustly regardless of $\sigma$. Nevertheless, it terminates at infeasible solutions with respect to the nonconvex quadratic constraints, even with $\sigma=1000$. The results in Table \ref{tab:hcn_shift05_comparison} also suggest that we could improve its final feasibility violation by further increasing the value of $\sigma$, yet at the price of using more computational resources due to the worse conditioning. As PCAL is observed to converge, we could infer from such deficiency of the QP-PCAL that the quadratic penalty function \eqref{eqn:CDFT quadratic penalty} is likely not exact for the nonlinear quadratic constraints in problem \eqref{eqn:CDFT discrete general}. 

\par Compared with QP-SCF and QP-PCAL, the WV and DASSP are free from the issues related to quadratic penalty and are effective in dealing with the nonlinear quadratic constraints. In particular, the proposed DASSP reformulates these constraints into linear ones by introducing an extra variable $Y$, and ensures the feasibility of $Y$ via closed-form projections. The discrepancy between $X$ and $Y$ is handled within an augmented Lagrangian framework, which does not require an excessively large penalty parameter. 
In the following experiments, we focus on the comparison between the proposed DASSP and the WV method.

\subsection{Comparison of computational efficiency}\label{subsec:efficiency test}

\par 
In this subsection, we compare the efficiency of the proposed DASSP and the WV method on a series of charge-transfer tests. We consider three molecular systems, \ce{HCN}, \ce{C2H4}, and \ce{CH3COCH3}, described at the same level of theory as in Section \ref{subsec:inexactness of quadratic penalty}. For \ce{HCN}, $(n,p)=(22,7)$ and the transfer is imposed between the \ce{C} $2p$ and \ce{N} $2p$ orbitals; for \ce{C2H4}, $(n,p)=(28,8)$ and the transfer is imposed between the 2p orbitals on the two carbon atoms; for \ce{CH3COCH3}, $(n,p)=(52,16)$ and the transfer is imposed between the $2p$ orbitals associated with the \ce{C} and \ce{O} atoms of the carbonyl group. Hence, we have $m=2$ for all the tests in this subsection. The shift amount $s$ is varied within $\{0.2,0.4,0.6,0.8,1.0\}$. Both the WV method and DASSP are initialized from the {\tt{minao}} guess provided by PySCF; for DASSP, we further set $\Lambda^{(0)}=0$. 
The maximum number of outer iterations in the WV method is set to $500$.

\par The numerical results are summarized in Table \ref{tab:shift_all}. 
\begin{table}[!t]
\centering
\caption{Comparison between DASSP and WV method on charge-transfer tests under different shift amounts $s$. The final energies, feasibility violations (Feas.), numbers of iterations (Iter.) and wall-clock times are listed. For the WV method, the values outside and inside the parentheses refer to the outer iteration numbers and the averaged inner iteration numbers, respectively.}
\label{tab:shift_all}
\begin{tabular}{c|c|cccr}
\hline\hline
\multicolumn{6}{c}{\ce{HCN} system: $n=22,\;p=7$} \\
\hline
$s$ & Algorithm & Energy ($E_h$) & Iter. & Feas. & Time (s)\\
\hline
\multirow{2}{*}{$0.2$} 
& WV    & $-9.208154276{\rm e}{+01}$ & $41(14.27)$ & $5.791{\rm e}{-09}$ & $9.06$ \\
& DASSP & $-9.208154275{\rm e}{+01}$ & $261$       & $2.838{\rm e}{-09}$ & $2.75$ \\
\hline
\multirow{2}{*}{$0.4$} 
& WV    & $-9.204088277{\rm e}{+01}$ & $71(10.34)$ & $4.313{\rm e}{-09}$ & $12.14$ \\
& DASSP & $-9.204088276{\rm e}{+01}$ & $232$       & $1.075{\rm e}{-09}$ & $2.52$ \\
\hline
\multirow{2}{*}{$0.6$} 
& WV    & $-9.197283505{\rm e}{+01}$ & $102(7.74)$ & $7.026{\rm e}{-09}$ & $13.90$ \\
& DASSP & $-9.197283504{\rm e}{+01}$ & $280$        & $2.049{\rm e}{-09}$ & $1.87$ \\
\hline
\multirow{2}{*}{$0.8$} 
& WV    & $-9.187696735{\rm e}{+01}$ & $76(8.17)$ & $5.660{\rm e}{-09}$ & $10.71$ \\
& DASSP & $-9.187696736{\rm e}{+01}$ & $303$       & $5.862{\rm e}{-10}$ & $2.05$ \\
\hline
\multirow{2}{*}{$1.0$} 
& WV    & $-9.175261870{\rm e}{+01}$ & $19(23.84)$ & $1.662{\rm e}{-09}$ & $4.87$ \\
& DASSP & $-9.175261870{\rm e}{+01}$ & $332$        & $8.110{\rm e}{-10}$ & $3.29$ \\
\hline\hline

\multicolumn{6}{c}{\ce{C2H4} system: $n=28,\;p=8$} \\
\hline
$s$ & Algorithm & Energy ($E_h$) & Iter. & Feas. & Time (s) \\
\hline
\multirow{2}{*}{$0.2$} 
& WV    & $-7.749661121{\rm e}{+01}$ & $19(19.47)$ & $3.883{\rm e}{-09}$ & $9.97$ \\
& DASSP & $-7.749661123{\rm e}{+01}$ & $340$        & $4.743{\rm e}{-09}$ & $4.91$ \\
\hline
\multirow{2}{*}{$0.4$} 
& WV    & $-7.746190903{\rm e}{+01}$ & $19(18.89)$ & $3.750{\rm e}{-09}$ & $9.54$ \\
& DASSP & $-7.746190902{\rm e}{+01}$ & $493$        & $6.683{\rm e}{-09}$ & $7.25$ \\
\hline
\multirow{2}{*}{$0.6$} 
& WV    & $-7.740360110{\rm e}{+01}$ & $19(21.21)$ & $4.508{\rm e}{-09}$ & $10.65$ \\
& DASSP & $-7.740360108{\rm e}{+01}$ & $364$        & $7.581{\rm e}{-09}$ & $5.32$ \\
\hline
\multirow{2}{*}{$0.8$} 
& WV    & $-7.732089387{\rm e}{+01}$ & $19(19.21)$ & $7.192{\rm e}{-09}$ & $9.78$ \\
& DASSP & $-7.732089389{\rm e}{+01}$ & $536$        & $1.256{\rm e}{-09}$ & $7.95$ \\
\hline
\multirow{2}{*}{$1.0$} 
& WV    & $-7.721248200{\rm e}{+01}$ & $19(18.53)$ & $5.354{\rm e}{-09}$ & $9.33$ \\
& DASSP & $-7.721248200{\rm e}{+01}$ & $494$        & $2.345{\rm e}{-10}$ & $7.34$ \\
\hline\hline

\multicolumn{6}{c}{\ce{CH3COCH3} system: $n=52,\;p=16$} \\
\hline
$s$ & Algorithm & Energy ($E_h$) & Iter. & Feas. & Time (s)\\
\hline
\multirow{2}{*}{$0.2$} 
& WV    & $-1.903782768{\rm e}{+02}$ & $17(20.24)$  & $1.141{\rm e}{-09}$ & $21.84$ \\
& DASSP & $-1.903782767{\rm e}{+02}$ & $354$        & $3.418{\rm e}{-09}$ & $18.29$ \\
\hline
\multirow{2}{*}{$0.4$} 
& WV    & $-1.903295420{\rm e}{+02}$ & $112(6.29)$  & $6.302{\rm e}{-09}$ & $57.17$ \\
& DASSP & $-1.903295420{\rm e}{+02}$ & $372$        & $5.223{\rm e}{-09}$ & $18.61$ \\
\hline
\multirow{2}{*}{$0.6$} 
& WV    & $-1.902483369{\rm e}{+02}$ & $39(12.41)$  & $5.329{\rm e}{-09}$ & $32.64$ \\
& DASSP & $-1.902483368{\rm e}{+02}$ & $332$        & $4.926{\rm e}{-10}$ & $16.89$ \\
\hline
\multirow{2}{*}{$0.8$} 
& WV    & $-1.901343269{\rm e}{+02}$ & $235(3.78)$  & $7.396{\rm e}{-09}$ & $80.03$ \\
& DASSP & $-1.901343270{\rm e}{+02}$ & $348$        & $1.401{\rm e}{-10}$ & $17.42$ \\
\hline
\multirow{2}{*}{$1.0$} 
& WV    & $-1.899868593{\rm e}{+02}$ & $101(5.42)$  & $5.848{\rm e}{-10}$ & $43.51$ \\
& DASSP & $-1.899868593{\rm e}{+02}$ & $283$        & $1.927{\rm e}{-09}$ & $13.81$ \\
\hline\hline
\end{tabular}
\end{table}
For all the tests here, both methods converge and attain essentially the same energies with comparable feasibility violations. For \ce{C2H4}, the efficiency of DASSP is slightly better than that of the WV method. But for both \ce{HCN} and \ce{CH3COCH3}, DASSP is found to yield prominent accelerations. 
Such superiority could be credited to the single-loop nature of DASSP. In contrast, although the WV method may require only a few outer iterations, its efficiency depends heavily on the inner dual Newton steps.

\subsection{Comparison of robustness}\label{subsec:charge-constrained DFT} 

\par In this subsection, we aim to numerically showcase the robustness advantage of DASSP over the WV method. We consider a representative charge-localization benchmark for CDFT, the water dimer cation \ce{(H2O)2+} in vacuum, following the setup in \cite{ahart2022cp2k}. 
The two water molecules are separated by a center-of-mass distance of 10 \AA. We describe the system under unrestricted Kohn-Sham DFT, using the PBE exchange-correlation functional \cite{perdew1996pbe} with the D3 dispersion correction \cite{grimme2010consistent} and the \texttt{def2-svp} basis set \cite{weigend2005balanced}, for which we have $(n,p)=(96,19)$. The charge-localization constraint ($m=1$) is imposed between two water fragments, based on the Hirshfeld charge partitioning scheme; please refer to Appendix \ref{appsec:nonconvex quadratic constraints} for details.  

\par We initialize both methods at random perturbations of a reference solution and compare their robustness using success rates as functions of the perturbation level. We first compute the reference solution by running DASSP with a tolerance of $10^{-8}$ and denote the final splitting variable by $Y_{\rm ref}\in\S^n$. For any perturbation level $\theta>0$, the initial point for both DASSP and WV methods is generated as $Y^{(0)}:=Y_{\text{ref}}+\theta N$, where $N \in \S^n$ is a symmetric random matrix sampled from the Gaussian distribution. We vary the perturbation level $\theta$ within $\{0,10^{-8},10^{-7},10^{-6},10^{-5},10^{-4},10^{-3},10^{-2},10^{-1},1\}$; for each value of $\theta$, $100$ independent trials are conducted using both methods. 
Throughout the tests in this subsection, we set $\Lambda^{(0)}=0$ for DASSP and $\bmu^{(0)}=0$ for the WV method. The maximum number of outer loops in the WV method is set to 500.

\par The numerical results are reported in Figure \ref{fig:success rate}, where the red and blue curves denote DASSP and the plain WV method, respectively.
\begin{figure}
    \centering
    \includegraphics[width=0.6\linewidth]{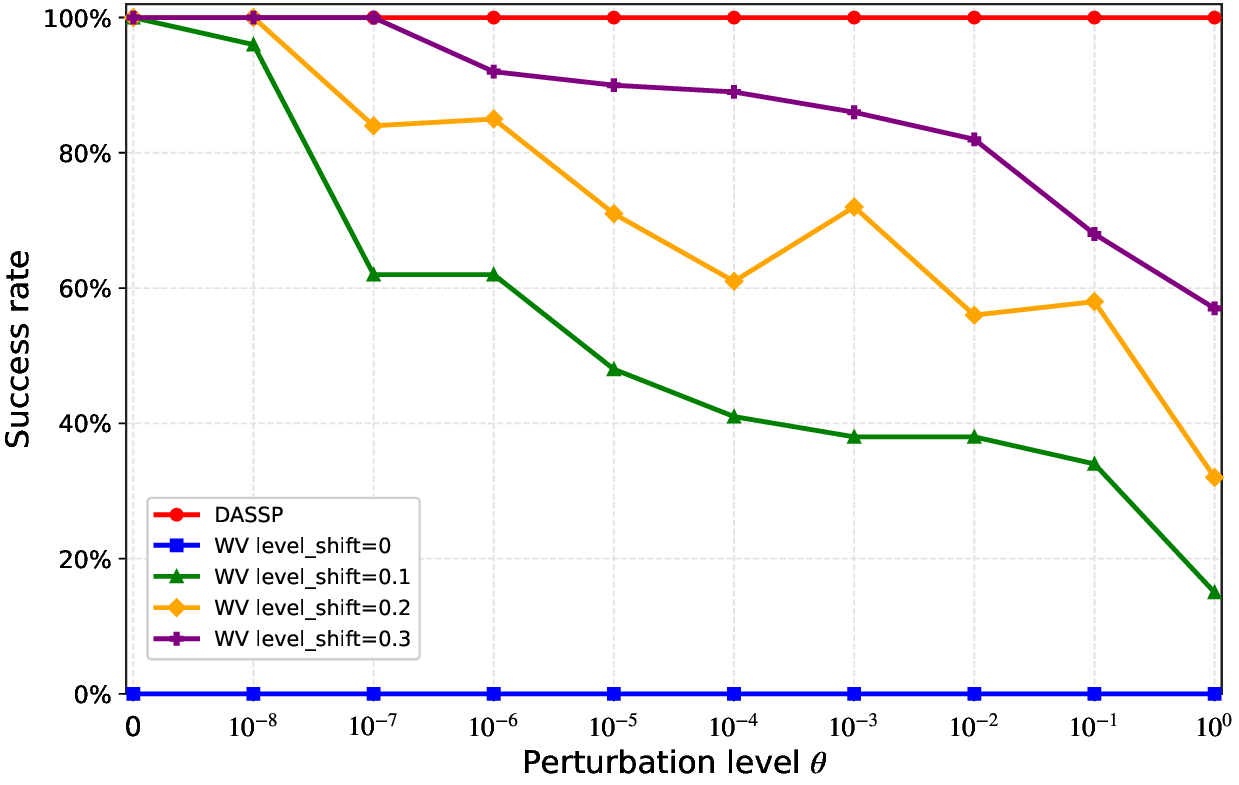}
    \caption{Success rates of DASSP and WV methods as functions of the perturbation level on the \ce{(H2O)2+} charge-localization test. The red curve with circle markers represents the results of DASSP, while the blue curve with square markers, green curve with triangle markers, yellow curve with diamond markers, and purple curve with plus sign markers represent those of the WV method with a level shift of 0, 0.1, 0.2, and 0.3, respectively.}
    \label{fig:success rate}
\end{figure}
The proposed DASSP is found to yield success rates of 100\% across the considered perturbation levels on this example. In sharp contrast, the plain WV method fails in all tests without any exception; in particular, it does not work even when the initialization is taken directly from the reference solution. 
As discussed in Section \ref{sec:introduction}, the convergence of the WV method could be problematic when the eigengap of the effective Hamiltonian is tiny. Therefore, we evaluate the eigengap of $F(X_{\rm ref},\mu_{\rm ref})$ at the reference solution, 
where $X_{\rm ref}\in\St_p(\R^n)$ satisfies $X_{\rm ref}X_{\rm ref}^\top=Y_{\rm ref}$ and $\mu_{\text{ref}}$ is obtained by solving the linear system \eqref{eqn:projection linear system}. 
The eigengap turns out to be extremely small, approximately $1.103\times 10^{-12}$. 

\par Intuitively, with such a tiny eigengap, the outer loop of the WV method can be unstable in selecting the lowest $p$ eigenvectors. Worse still, a tiny eigengap is also harmful to the inner Newton steps, as the required Jacobian relies on its reciprocal \cite{wu2005direct}. We corroborate these explanations by monitoring the eigengap and consecutive subspace distance,  $\snorm{X^{(k)}(X^{(k)})^\top-X^{(k+1)}(X^{(k+1)})^\top}_{\F}$, of the WV method, using either random initialization with a perturbation level of $\theta=10^{-8}$ described above or the \texttt{minao} guess provided by PySCF (whose distance from the reference solution is $1.649$). We collect the results in Figure \ref{fig:wv_eigengap_subspace}. 
\begin{figure}[htp]
\centering
\begin{subfigure}[t]{0.48\textwidth}
    \centering
    \includegraphics[width=\linewidth]{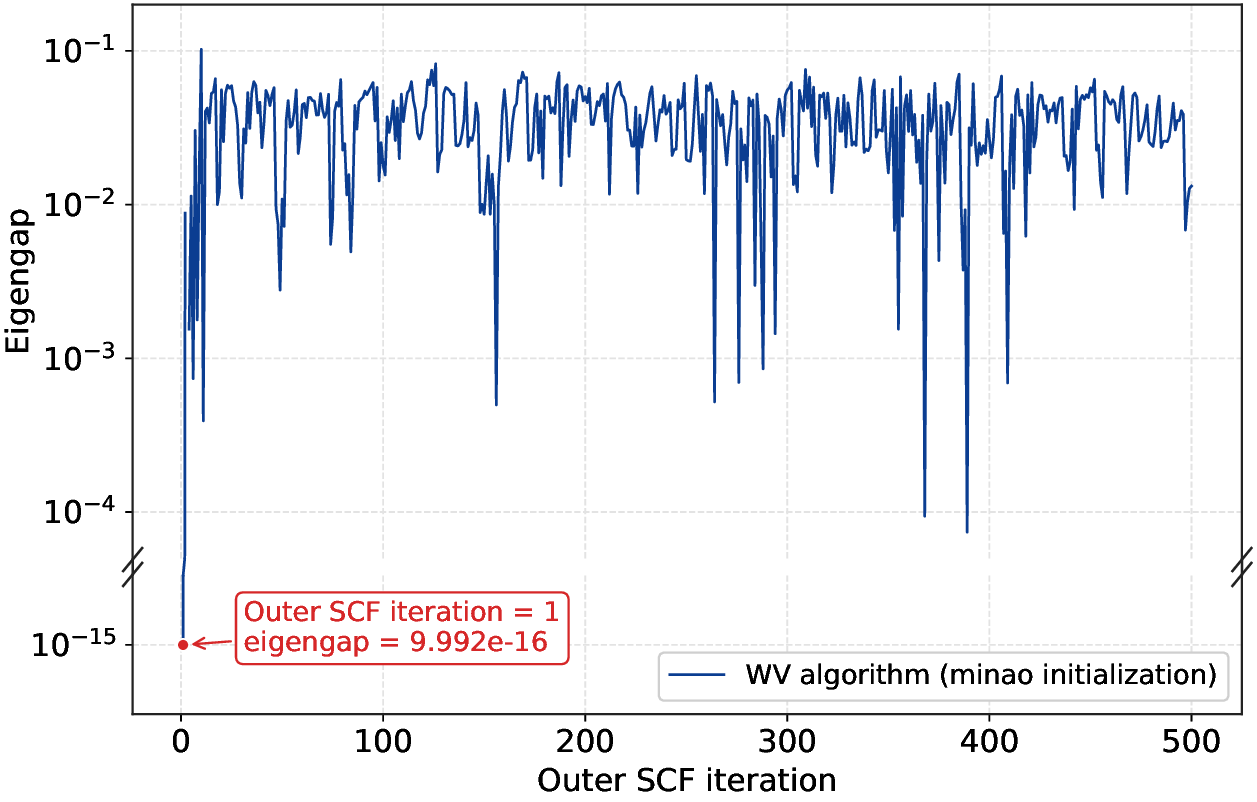}
    \caption{Eigengap, \texttt{minao}.}
    \label{fig:wv_minao_eigengap}
\end{subfigure}
\begin{subfigure}[t]{0.48\textwidth}
    \centering
    \includegraphics[width=\linewidth]{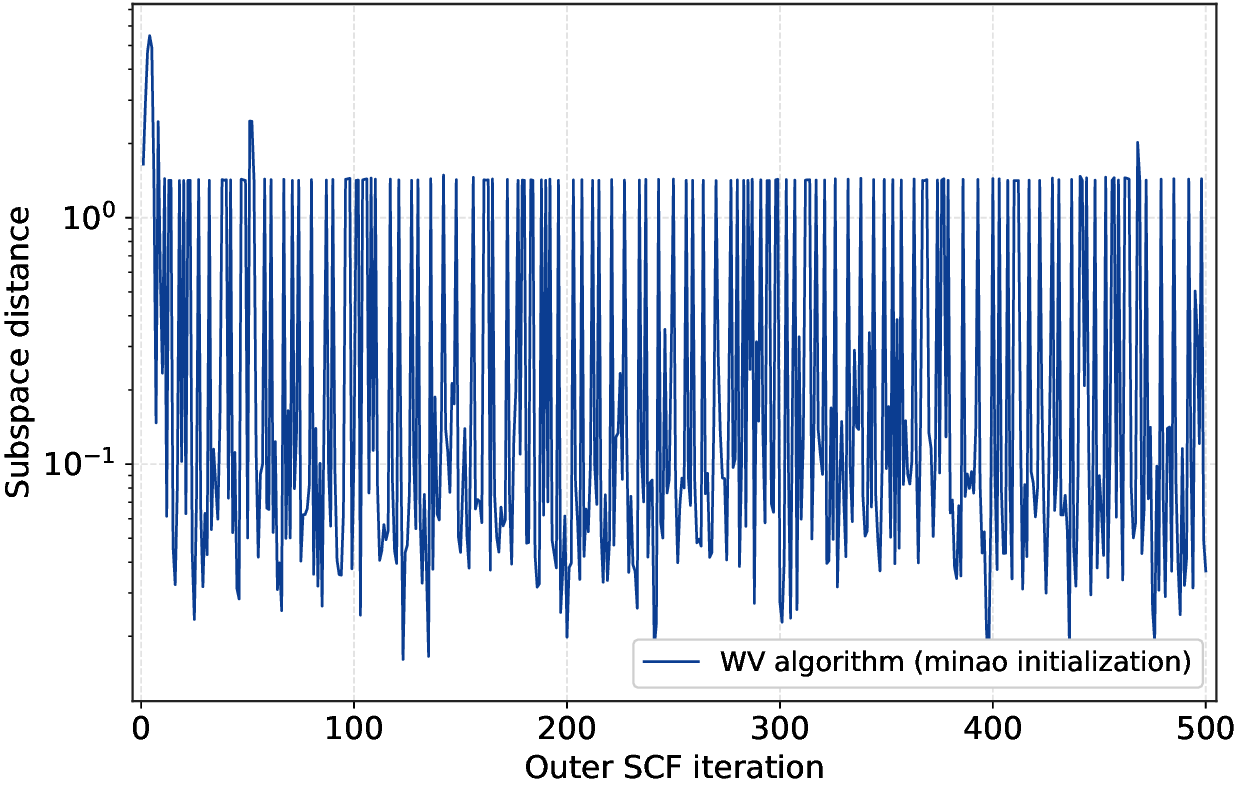}
    \caption{Subspace distance, \texttt{minao}.}
    \label{fig:wv_minao_subspace}
\end{subfigure}\\[0.2cm]
\begin{subfigure}[t]{0.48\textwidth}
    \centering
    \includegraphics[width=\linewidth]{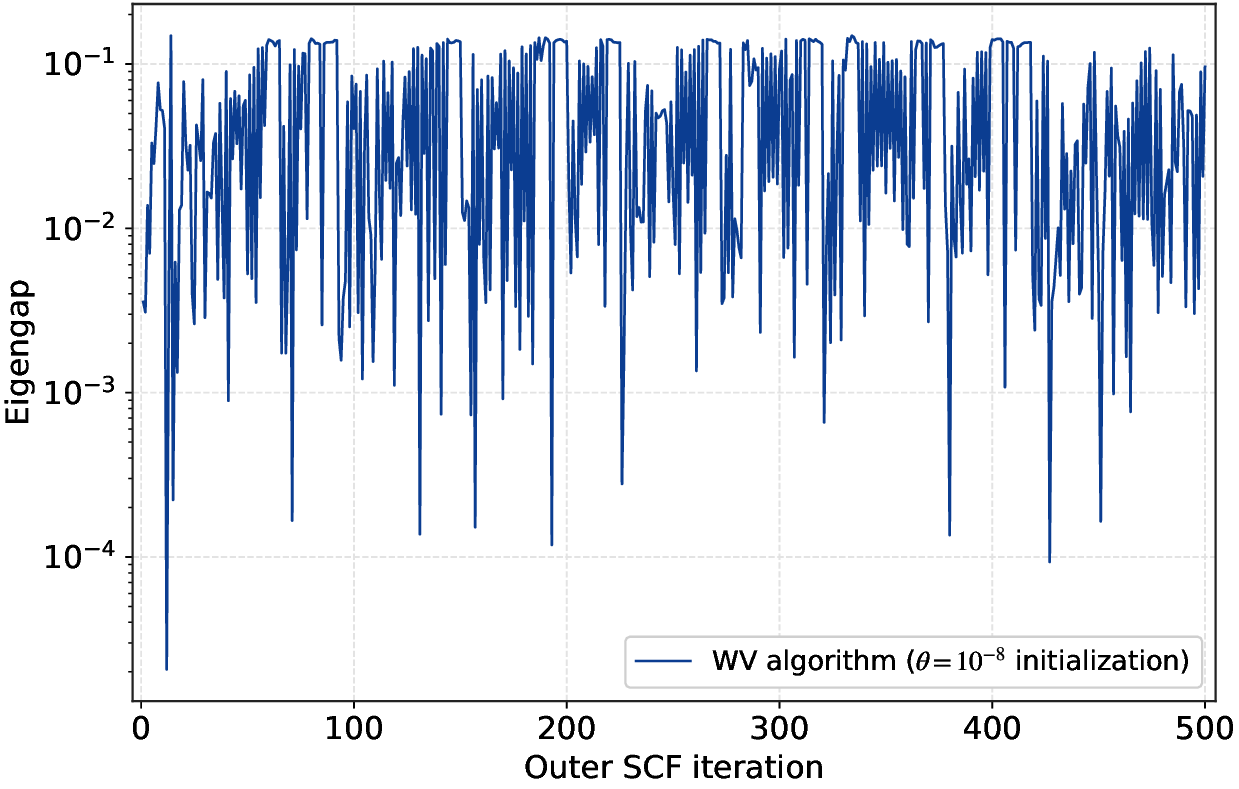}
    \caption{Eigengap, random.}
    \label{fig:wv_sigma_eigengap}
\end{subfigure}
\begin{subfigure}[t]{0.48\textwidth}
    \centering
    \includegraphics[width=\linewidth]{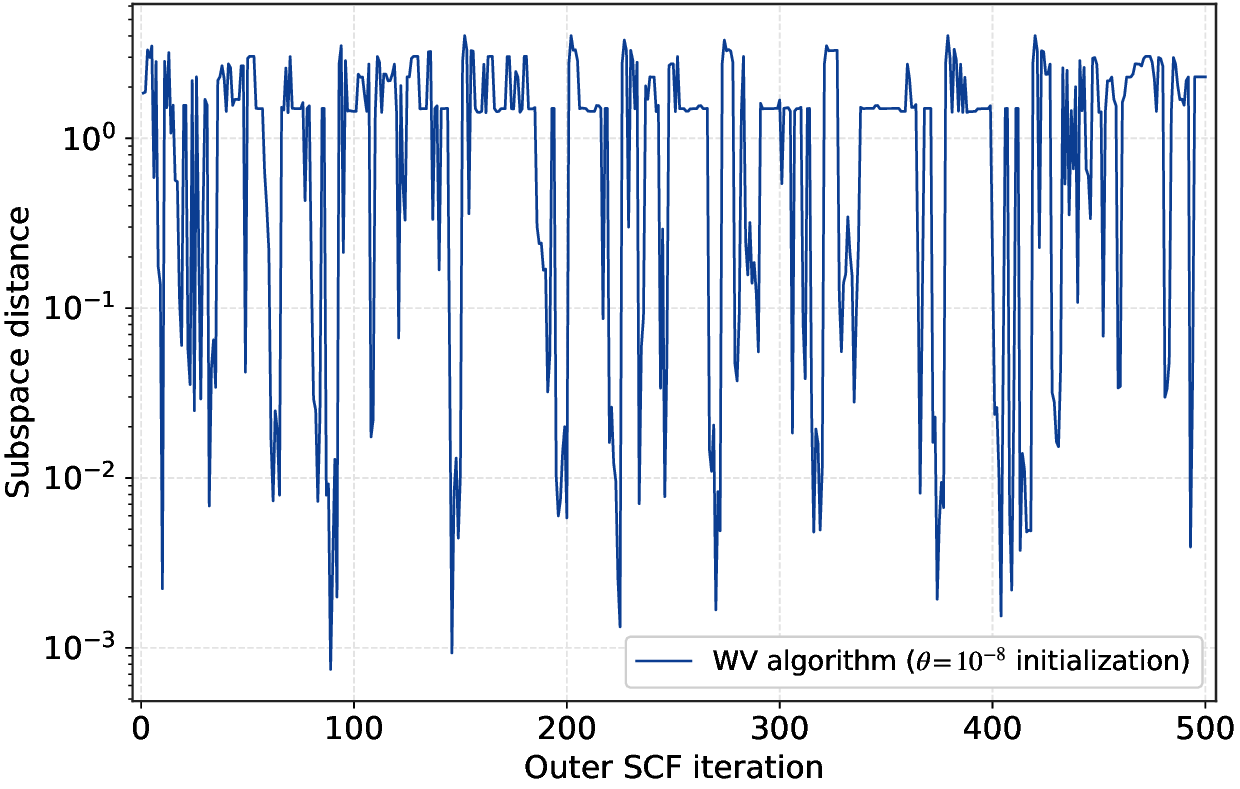}
    \caption{Subspace distance, random.}
    \label{fig:wv_sigma_subspace}
\end{subfigure}
\caption{Evolutions of the eigengap of the effective Hamiltonian and consecutive subspace distance of the WV algorithm with either \texttt{minao} guess (panels (a) and (b)) or randomly perturbed initialization (panels (c) and (d)).}
\label{fig:wv_eigengap_subspace}
\end{figure}
Under both initializations, 
the eigengap of the effective Hamiltonian fluctuates severely over the outer iterations and becomes small from time to time, and the subspace distance exhibits persistent oscillations. These observations indicate that the WV method fails to stabilize subspaces, even when the initial point is already sufficiently close to the reference solution.

\par We acknowledge that the convergence proof of DASSP also involves an eigengap-type Assumption \ref{assumption} \ref{asp:eigengap} on the matrix $A^{(k)}$ in the $X$-subproblem \eqref{eqn:X subproblem}, which is not the same as the effective Hamiltonian in the WV method. Although we cannot remove this assumption in the current analysis, we present here a numerical investigation. To stay in line with the convergence analysis in Section \ref{sec:convergence analysis}, we consider DASSP using $\delta=10^{-10}$, starting with random initialization with a perturbation level of $\theta=1$ or the \texttt{minao} guess by PySCF. The numerical results are presented in Figure \ref{fig:DASSP_eigengap}. 
\begin{figure}[htb]
\centering

\begin{subfigure}[t]{0.48\textwidth}
    \centering
    \includegraphics[width=\linewidth]{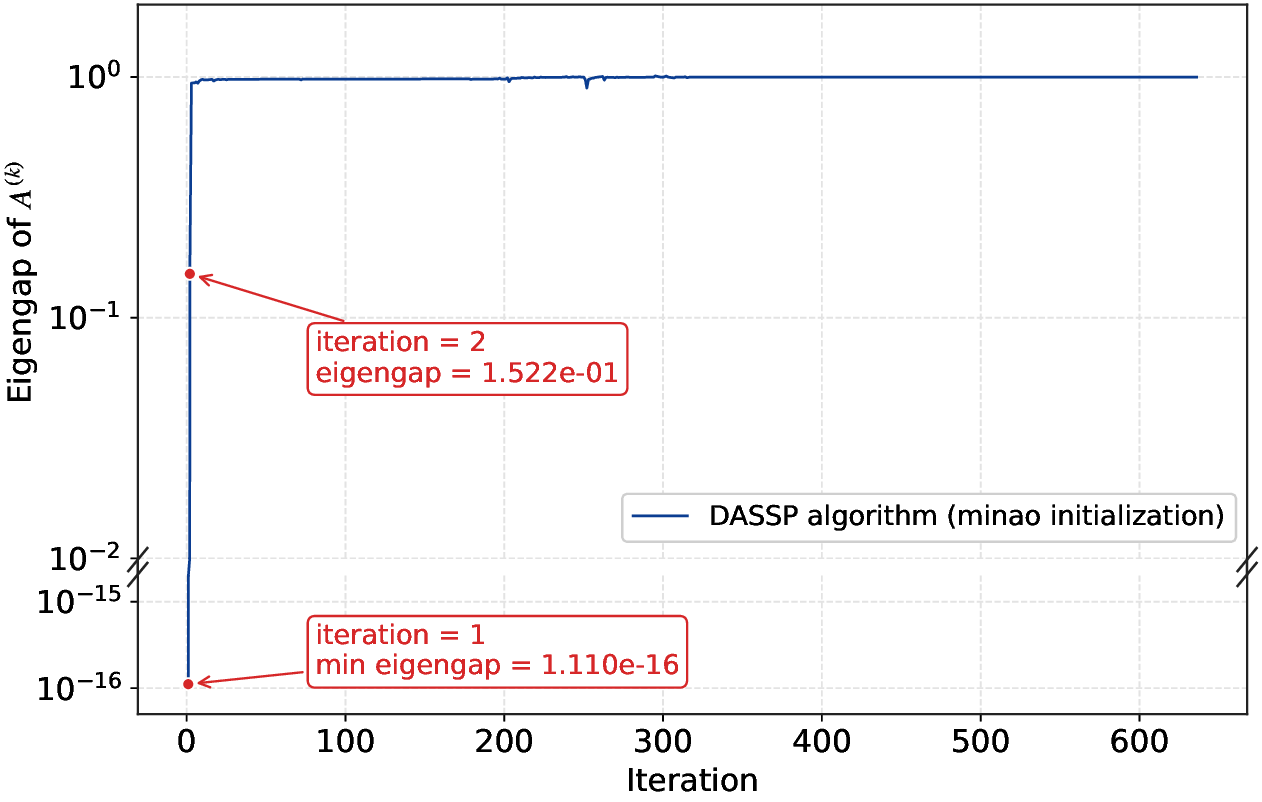}
    \caption{Eigengap, \texttt{minao}.}
    \label{fig:DASSP_minao_eigengap}
\end{subfigure}
\begin{subfigure}[t]{0.48\textwidth}
    \centering
    \includegraphics[width=\linewidth]{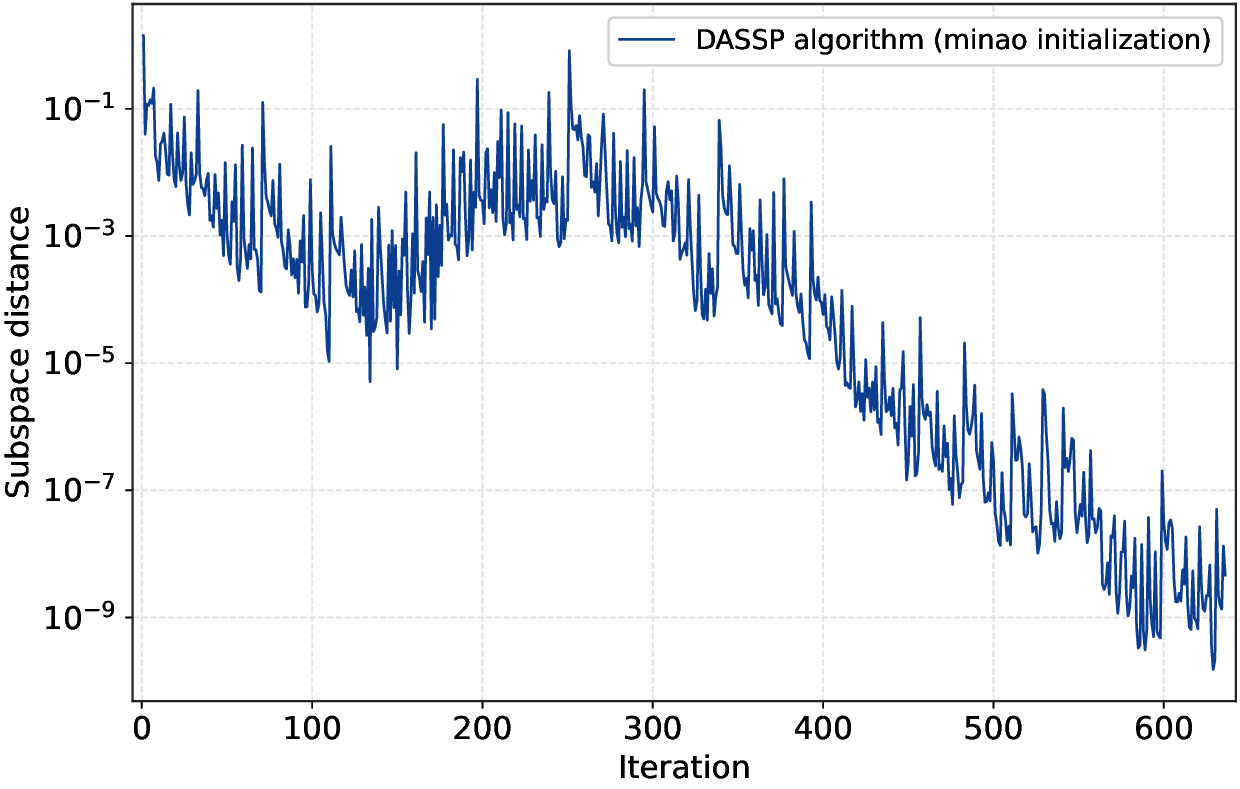}
    \caption{Subspace distance, \texttt{minao}.}
    \label{fig:DASSP_minao_subspace}
\end{subfigure}\\[0.2cm]
\begin{subfigure}[t]{0.48\textwidth}
    \centering
    \includegraphics[width=\linewidth]{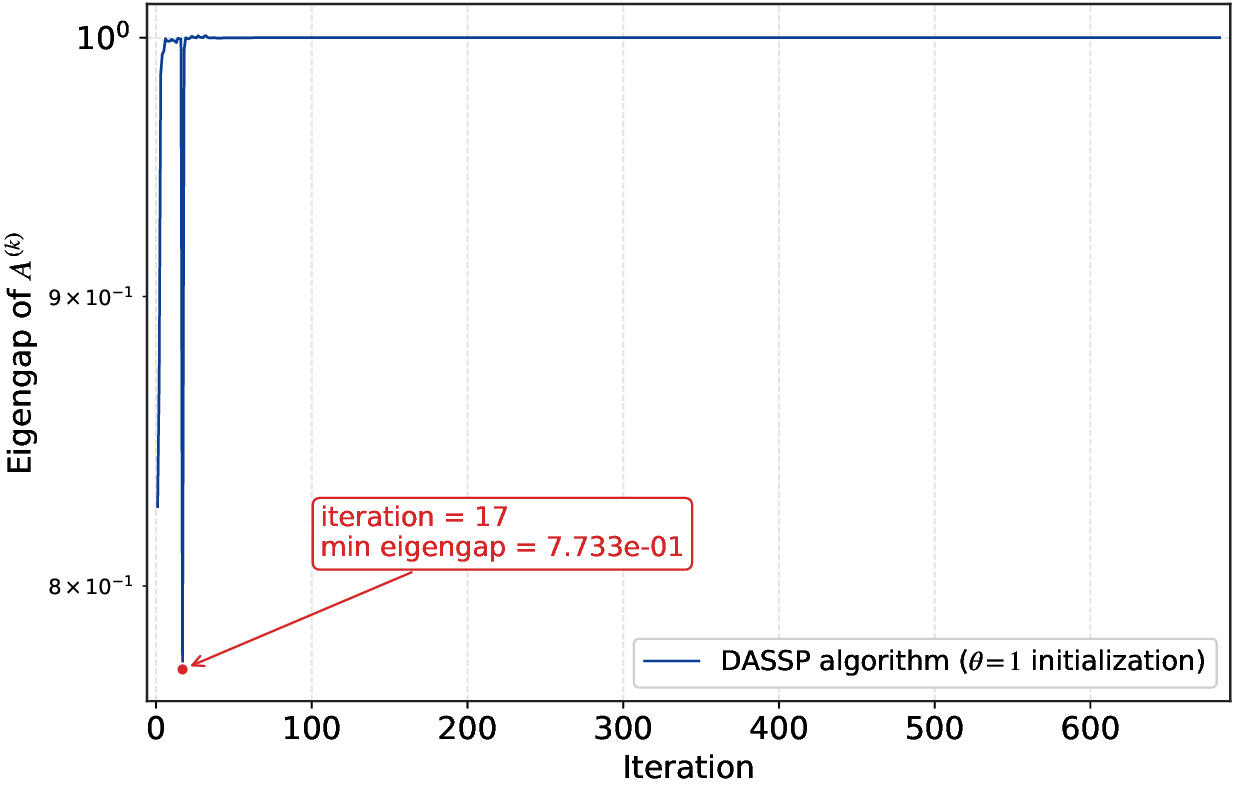}
    \caption{Eigengap, random.}
    \label{fig:DASSP_theta_eigengap}
\end{subfigure}
\begin{subfigure}[t]{0.48\textwidth}
    \centering
    \includegraphics[width=\linewidth]{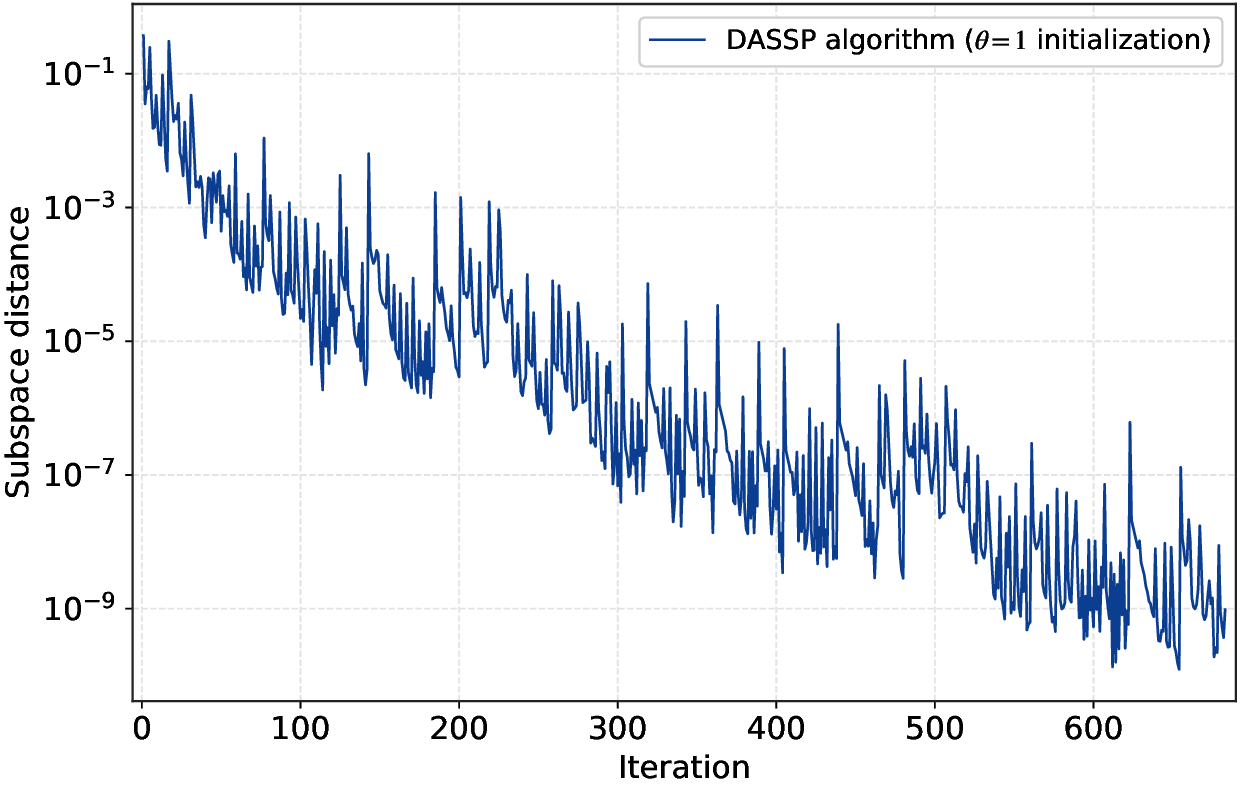}
    \caption{Subspace distance, random.}
    \label{fig:DASSP_theta_subspace}
\end{subfigure}

\caption{Evolutions of the eigengap of the matrix $A^{(k)}$ and consecutive subspace distance of DASSP with either \texttt{minao} guess (panels (a) and (b)) or randomly perturbed initialization (panels (c) and (d)).}
\label{fig:DASSP_eigengap}
\end{figure}
These plots show that the eigengap of $A^{(k)}$ quickly rises to a positive plateau and remains bounded away from zero, even though the initial eigengap of $A^{(0)}$ may vanish (see panel (a) in Figure \ref{fig:DASSP_eigengap}). This suggests that Assumption \ref{assumption} \ref{asp:eigengap} is not restrictive in practice. By virtue of the behavior shown in Figure \ref{fig:DASSP_eigengap}, DASSP converges steadily to the reference solution. 

\par In the SCF literature (e.g., \cite{bai2022sharp}), a commonly used strategy to mitigate the issues related to small eigengap is to introduce a level shift into the effective Hamiltonian, which artificially enlarges the eigengap between the $p$-th and $(p+1)$-th smallest eigenvalues and may improve the stability of the SCF iteration. In PySCF, this can be achieved by setting the tag \texttt{level\_shift}. We also test the WV method with the amount of level shift within $\{0.1,0.2,0.3\}$, under the same perturbed initialization strategy above. The success rates are also shown in Figure \ref{fig:success rate}. One can see that level shifting indeed improves the robustness of the WV method, and as the amount of level shift increases, the success rate is getting better. Nevertheless, the success rate of the WV method still exhibits a clear decreasing trend as the perturbation level grows. In addition, as pointed out in \cite{bai2022sharp}, increasing the amount of level shift would degrade the efficiency of the algorithm. In comparison, DASSP offers both efficiency and robustness under the default parameter setting. 


\section{Conclusions}\label{sec:conclusions}

\par In this paper, we study the optimization problem arising from CDFT. Compared with that for ground-state DFT calculations, it is more challenging due to the presence of both the Stiefel manifold and nonconvex quadratic constraints. By exploiting the inherent rotation invariance, we decouple the two groups of constraints by introducing an additional nonlinear subspace alignment constraint. For the reformulated problem, we propose a single-loop damped ADMM termed DASSP, which, to the best of our knowledge, is the first algorithm with rigorous convergence guarantees for CDFT calculations. Each iteration of DASSP consists of a spectral minimization step, a projected-gradient step, and a damped dual update. Numerical tests have been conducted on synthetic and realistic charge-transfer and charge-localization CDFT calculations. The obtained results demonstrate that DASSP is effective in dealing with the nonconvex quadratic constraints, and exhibits comparable performance or substantial acceleration over state-of-the-art methods without compromising robustness.

\par Several directions merit further investigation. First, it is of theoretical interest to the broader ADMM community to study the convergence properties in the undamped case ($\delta=0$). Our numerical results show that the undamped version converges steadily and is efficient, while the analysis in this paper only applies to the damped regime ($\delta>0$). Second, it is  preferable to investigate an adaptive way to choose the parameters in DASSP for general settings. Third, it would be practically meaningful to apply or adapt DASSP to other CDFT settings. For example, in noncollinear calculations, one could impose constraints on local atomic magnetic moments to study the potential energy surfaces of magnetic materials. The underlying problems may involve larger numbers of nonconvex quadratic constraints than the considered charge-transfer or charge-localization problems considered here. We could thus anticipate a more prominent advantage of DASSP over the double-loop method, which relies on a dual scheme for inner iterations. 
Finally, it is promising to integrate DASSP for first-principles calculations into machine-learning-assisted quantum chemistry and quantum physics workflows. For challenging applications such as magnetic materials and charge-transfer processes, large, high-quality first-principles datasets are essential but remain scarce. The efficiency and robustness of DASSP suggest its potential for generating high-quality CDFT datasets in such settings.

\appendix

\section{Discretized formulation of CDFT }\label{appsec:UKS functional}

\par We describe how the continuous CDFT formulation \eqref{eqn:CDFT continuous general} is discretized into the finite-dimensional optimization model \eqref{eqn:CDFT discrete general} used throughout the paper. For simplicity, we take restricted Kohn-Sham DFT (RKS-DFT) in closed-shell collinear calculations as an example.

\par In collinear calculations, spinor orbitals are assumed to be of the form $\psi_i=(\varphi_i,0)^\top$ (spin-up) or $\psi_i=(0,\varphi_i)^\top$ (spin-down), where $\varphi_i$ is a spatial orbital. Within RKS-DFT, each spatial orbital is doubly occupied by electrons of opposite spin with identical spatial functions. Let $\{\chi_\mu\}_{\mu=1}^{\nb}$ be a set of real atomic-orbitals, where $\nb\in\N$ denotes the number of basis functions. For brevity, we assume here (and in Appendix \ref{appsec:nonconvex quadratic constraints}) that the atomic orbitals are orthonormal in the $L^2$ inner product (which can be enforced using the L\"owdin transformation). The spatial orbital is expressed as a linear combination of atomic orbitals:
$$\varphi_i=\sum_{\mu=1}^{\nb}X_{\mu i}\chi_\mu,\quad i=1,\ldots,N_{\rm o}~~\text{with}~~N_{\rm o}:=N/2.$$
Here, $N_{\rm o}$ denotes the number of occupied spatial orbitals and $X\in\R^{\nb\times N_{\rm o}}$ denotes the orbital coefficient matrix. The orthonormality condition on the spinor orbitals then reduces to the constraint $X^\top X=I_{N_{\rm o}}$ on $X$. The reduced one-body density matrix is defined as $Y:=XX^\top\in\S^{\nb}$. The spin densities are thus given by
\begin{equation}
    \rho^{\uparrow,\uparrow}(\rr)=\rho^{\downarrow,\downarrow}(\rr)=\sum_{1\le\mu,\nu\le\nb}Y_{\mu\nu}\chi_\mu(\rr)\chi_\nu(\rr)\quad\text{and}\quad\rho^{\uparrow,\downarrow}=\rho^{\downarrow,\uparrow}\equiv0.
    \label{eqn:RKS spin densities}
\end{equation}

\par In the chosen atomic-orbital basis, the discretized RKS-DFT energy functional is
$$f(Y)=2\trace(hY)+2\trace(J(Y)Y)+E_{\xc}(Y)+E_{\rm nn},$$
where $h=[h_{\mu\nu}]\in \S^{\nb}$ is the discretized one-body Hamiltonian, 
$$h_{\mu\nu}:=\int_{\R^3}\chi_\mu(\rr)\lrbracket{-\frac12\Delta+V_{\rm ext}(\rr)}\chi_\nu(\rr)\dd\rr,\quad1\le\mu,\nu\le\nb,$$
with $V_{\rm ext}$ the external potential (e.g., induced by nuclei), $J: \S^{\nb}\to \S^{\nb}$ is the self-adjoint Coulomb operator, defined as 
\begin{equation*}
[J(Y)]_{\mu\nu}:=\sum_{1\leq \sigma,\tau\leq \nb}g_{\mu\nu\sigma\tau}Y_{\sigma\tau},~~1\leq \mu,\nu\leq \nb,
\end{equation*}
with the two-body integrals
\begin{equation*}
g_{\mu\nu\sigma\tau}:=\int_{\R^3}\int_{\R^3}\frac{\chi_\mu(\rr)\chi_\nu(\rr)\chi_\sigma(\rr')\chi_\tau(\rr')}{\norm{\rr-\rr'}}\dd \rr \dd \rr',~~1\leq \mu,\nu,\sigma,\tau\leq \nb,
\end{equation*}
$E_{\xc}:\S^{\nb}\to \R$ is the exchange-correlation functional, and $E_{\text{nn}}$ is the nuclear repulsion term. 

\par In the chosen atomic-orbital basis, equation \eqref{eqn:RKS spin densities} implies that the additional density constraints in CDFT can be discretized into the form $\trace(X^\top W_jX)=b_j$, where in particular 
\begin{equation}
    [W_j]_{\mu\nu}=\int_{\R^3}\chi_\mu(\rr)\big(w_j^{\uparrow,\uparrow}(\rr)+w_j^{\downarrow,\downarrow}(\rr)\big)\chi_\nu(\rr)\dd\rr,\quad1\le\mu,\nu\le\nb.
    \label{eqn:discretization of weight function}
\end{equation}

\par In summary, the finite-dimensional discretization of the CDFT problem \eqref{eqn:CDFT continuous general} within RKS-DFT takes the form \eqref{eqn:CDFT discrete general} with $n=\nb$ and $p=N_{\rm o}$. We remark that the extensions to the unrestricted case (UKS-DFT) and noncollinear case are straightforward. In UKS-DFT, the spatial orbitals in the spin-up and spin-down channels are independently optimized and may differ. In the noncollinear case, both components of the spinor orbitals are nonzero and no fixed global spin quantization axis exists.

\section{Construction of nonconvex quadratic constraints}\label{appsec:nonconvex quadratic constraints}

\par\noindent\textbf{Charge-transfer problems.} This construction is used in the experiments in Sections \ref{subsec:inexactness of quadratic penalty} and \ref{subsec:efficiency test}. Following standard population analysis in CDFT, fragment charges are measured in a localized-orbital basis.  Let $C_{\text{loc}}\in \R^{\nb\times\nb}$ denote the localized-orbital transformation satisfying $C_{\text{loc}}^{\top}C_{\text{loc}}=I_{\nb}$, which can be constructed by the meta-L\"owdin orthogonalization routine in PySCF. Let $\calI_A$ and $\calI_B$ denote the sets of localized orbitals associated with fragments $A$ and $B$, respectively. We define the corresponding population matrices by
\begin{equation*}
W_A:=\sum_{\ell \in \calI_A}c_{\ell}c_{\ell}^{\top}\in \S^{\nb},\qquad W_B:=\sum_{\ell\in \calI_B}c_{\ell}c_{\ell}^{\top}\in \S^{\nb},
\end{equation*}
where $c_{\ell}\in \R^{\nb}$ is the $\ell$-th column of $C_{\text{loc}}$. Let $X_{\text{ref}}\in \R^{n\times p}$ be a reference solution obtained from, e.g., a ground-state DFT calculation and let $s>0$ denote the prescribed charge-transfer amount.
Then we impose the charge-transfer constraints as follows:
\begin{equation*}
\trace{(X^{\top}W_A X)}=b_A:=\trace{(X^{\top}_{\text{ref}}W_A X_{\text{ref}})}+s,\quad  \trace{(X^{\top}W_B X)}=b_B:=\trace{(X^{\top}_{\text{ref}}W_B X_{\text{ref}})}-s.
\end{equation*}
That is, we enforce a charge transfer of size $s$ from fragment $B$ to fragment $A$. 

\medskip 

\par\noindent\textbf{Charge-localization problems.} This applies to the experiments in Section \ref{subsec:charge-constrained DFT}. Let $L$ and $R$ denote the two water fragments, corresponding to the left and right molecules, respectively. For each atom $a$, let $\rho^0_{a}$ be the promolecular atomic density centered at the nuclear position $\RR_a\in \R^3$, obtained from a DFT calculation with atom $a$ in isolation. In our implementations, we use the same exchange-correlation functional and basis set for all atoms. The Hirshfeld weight function is then defined by 
\begin{equation*}
w(\rr):=\frac{\sum_{a\in L}\rho_{a}^{0}(\rr-\RR_a)-\sum_{a\in R}\rho_{a}^{0}(\rr-\RR_a)}{\sum_{a}\rho_{a}^{0}(\rr-\RR_a)}.
\end{equation*}
Here, the denominator is the total promolecular density. Following the standard Hirshfeld implementation \cite{ahart2022cp2k}, when this denominator is smaller than $10^{-12}$, we set $w(\rr)=0$ for numerical stability. The Hirshfeld charge constraint takes the form $\trace{(X^{\top}WX)}=N_{c}$, where the matrix $W$ is defined as in equation \eqref{eqn:discretization of weight function} with $w^{\uparrow,\uparrow}=w^{\downarrow,\downarrow}=w$. For the water dimer cation, we set $N_c=1$, which corresponds to a one-electron charge-difference constraint between the two water fragments.

\normalem
\bibliography{ref}
\bibliographystyle{plain}

\end{document}